\documentclass[12pt, showkeys]{article}
\usepackage{amsmath, amssymb, ascmac}
\usepackage{bm}
\usepackage[abbrev]{amsrefs}
\usepackage{enumerate}
\usepackage{amsthm}
\usepackage{comment}
\usepackage[title]{appendix}
\usepackage{color}
\usepackage[dvipdfmx]{graphicx}
\usepackage[subrefformat=parens]{subcaption}

\usepackage{hyperref}

\usepackage[top=30truemm, bottom=30truemm, left=20truemm, right=20truemm]{geometry}

\DeclareMathOperator{\var}{Var}
\DeclareMathOperator{\cov}{Cov}
\DeclareMathOperator{\dom}{Dom}

\newtheorem{theorem}{\bf Theorem}
\newtheorem{lemma}{\bf Lemma}[section]

\newtheorem{remark}{\rm REMARK}[section]

\newtheorem{proposition}{\bf Proposition}[section]
\newtheorem{assumption}{\bf Assumption}

\title{{\huge A central limit theorem for \\the stochastic cable equation}}
\author{{\large Soma Nishino}}
\date{}
\begin{document}
\maketitle

\begin{abstract}
We study one-dimensional nonlinear stochastic cable equations driven by a multiplicative space-time white noise.
Using the Malliavin--Stein method, we prove a central limit theorem for the spatial average of the solution.
The convergence is established in the total variation distance with mild conditions.
We also establish a functional central limit theorem with a technical assumption.
Furthermore, we show that this assumption holds in a special case.

\bigskip
\noindent {\bf Keywords:}
Stochastic cable equation, Central limit theorem, Malliain calculus, Stein's method
\footnote[0]{2020 Mathematics Subject Classification: 
Primary 60H15; Secondary 60F05, 60H07.}
\end{abstract}


\section{Introduction and main results}
In this paper, we consider the nonlinear stochastic cable equation
\begin{equation}\label{cableeq}
    \frac{\partial u}{\partial t} = \frac{\beta}{2} \frac{\partial^2 u}{\partial x^2} - \alpha u + \sigma(u) \dot{W}
\end{equation}
on $[0,T] \times [0,L]$ for some constants $\alpha \in \mathbb{R}$ and $\beta > 0$, where $T > 0$ is fixed, $\dot{W}$ is a space-time white noise on $[0,T] \times [0,L]$, with initial condition $u_0(x) = 1$, Neumann, Dirichlet, or periodic boundary conditions.
We assume the coefficient $\sigma$ is global Lipschitz.

According to Walsh~\cite{walsh}, the above equation admits a unique mild solution, which is adapted to the filtration generated by $W$ and satisfies the condition $E[u(t,x)^2] < \infty$.
The mild solution satisfies the following equation
\begin{equation}\label{mildsol}
    u(t,x) = \int_0^L u_0(y) G_t(x,y) \ \mathrm{d} y + \int_0^t \int_0^L G_{t-s}(x,y) \sigma(u(s,y)) \ W(\mathrm{d}s, \mathrm{d}y)
\end{equation}
where in the right hand side the stochastic integral is in the sense of It\^o--Walsh, and $G$ is the Green's function for the cable equation \eqref{det_cableeq}.

We study the large $L$ asymptotics of the spatial average $F_L(t)$ of the solution, given by
\begin{equation}\label{flt}
    F_L(t) := \frac{1}{L} \int_0^L \left\{ u(t,x) - E[u(t,x)] \right\} \ \mathrm{d} x .
\end{equation}

For two random variables $X$ and $Y$, the total variation distance is defined as
\begin{equation*}
    d_{\text{TV}}(X,Y) = \sup_{B \in \mathcal{B}(\mathbb{R})} | P(X \in B) - P(Y \in B) |,
\end{equation*}
where the supremum is taken over all sets $B$ in the Borel $\sigma$-algebra $\mathcal{B}(\mathbb{R})$.
We denote by $d_{\text{TV}}(F, \mathcal{N}(0,1))$ the total variation distance between the law of a random variable $F$ and the standard normal distribution.

We are now ready to state the first main result of this paper.
\begin{theorem}
    \label{thm1}
    Suppose that $\sigma(1) \neq 0$.
    Then, for every $t > 0$ there exists a real number $c = c(t) > 0$ such that for all $L \geq 1$,
    \begin{equation}
        \label{d_conv}
        d_{\mathrm{TV}} \left( \frac{F_L(t)}{\sqrt{\var(F_L(t))}} , \mathcal{N}(0,1) \right) \leq \frac{c}{\sqrt{L}} .
    \end{equation}
\end{theorem}

\begin{remark}
    Condition $\sigma(1) \neq 0$ guarantees that $\var(F_L(t)) > 0$.
    This follows from \eqref{eq_guarantees1}, \eqref{eq_guarantees2}, \eqref{eq_guarantees3}, \eqref{eq_guarantees4}, \eqref{dirichlet_rep} and \eqref{periodic_rep}.
\end{remark}

We assume the following technical condition, which is required for the proof of the functional central limit theorem.
\begin{assumption}
    \label{assump}
    There exists some nonnegative valued measurable function $f_{\sigma}$ such that for any $T > 0$, $f_{\sigma} \in L^1([0,T])$ and for all $t \in [0,T]$,
    \begin{equation}\nonumber
        \lim_{L \to \infty} \frac{1}{L} \int_0^L E[\sigma(u(t,x))^2] \ \mathrm{d} x = f_{\sigma}(t) .
    \end{equation}
\end{assumption}

We state the following functional central limit theorem.
\begin{theorem}
    \label{thm2}
    Fix $T > 0$. Suppose that Assumption \ref{assump} holds.
    Then, as $L \to \infty$,
    \begin{equation}\nonumber
        \left( \sqrt{L} F_L(t) \right)_{t \in [0,T]} \to \left( \int_0^t e^{-\alpha (t-s)} \sqrt{f_{\sigma}(s)} \ \mathrm{d} W_s \right)_{t \in [0,T]}
    \end{equation}
    where $f_{\sigma}(t)$ is the limit in Assumption~\ref{assump} and $W = \{ W_s \}_{s \in [0,T]}$ denotes a standard one-dimensional Brownian motion, and the convergence is in law on the space of continuous functions $C([0,T])$.
\end{theorem}

In recent years, considerable attention has been paid to the asymptotic behavior of spatial averages for solutions to stochastic partial differential equations.
This area of inquiry was first explored by Huang, Nualart, and Viitasaari in~\cite{huang_nualart_viitasaari}.
Utilizing the Malliavin--Stein method (see~\cite{nourdin_peccati}), they proved a central limit theorem, along with its functional counterpart, for the one-dimensional nonlinear stochastic heat equation with space-time white noise.
Following their seminal work, analogous central limit theorems for various types of stochastic partial differential equations have been established.
For example, Huang, Nualart, Viitasaari, and Zheng studied the $d$-dimensional stochastic heat equation with colored noise~\cite{huang_nualart_viitasaari_zheng}.
There has been much research on the stochastic heat equation under different settings; see, for example, \cite{assaad_nualart_tudor_viitasaari,balan_yuan,chen_khoshnevisan_nualart_pu,khoshnevisan_nualart_pu,kim_yi,nualart_xia_zheng,nualart_zheng,pu}.
Similar results are also known for the stochastic wave equation; see, for instance, \cite{delgado_nualart_zheng} for the one-dimensional case, \cite{guerrero_nualart_zheng} for the two-dimensional case, and \cite{ebina,ebina_2} for the three and higher dimensions.
For other types of SPDEs, Liu and Shen studied a central limit theorem for the solution to an SPDE related to a pseudo-differential operator that generates a stable-like process~\cite{liu_shen}.

The previous study most relevant to our setting is that of Pu~\cite{pu}.
Pu considered the one-dimensional stochastic heat equation with boundary conditions on a bounded interval $[0,L]$ and analyzed the solution using its Wiener chaos decomposition. 
In particular, we remark that Assumption \ref{assump} is necessary for the proof of the functional central limit theorem when considering SPDEs on a bounded domain, as in the present work and that of Pu~\cite{pu}.

Assumption~\ref{assump} is verified in Section~\ref{on_assump} for the case $\sigma(u) = \sigma_1 u + \sigma_0$, where $\sigma_1$ and $\sigma_0$ are constants.
The argument relies on the Wiener chaos decomposition of the solution to the stochastic cable equation~\eqref{cableeq}.
The same method was used in~\cite{pu} to establish Assumption~\ref{assump} for the stochastic heat equation with $\sigma(u) = u$.
In~\cite{huang_nualart_viitasaari} and related works, since the solution to the SPDE under consideration is spatially stationary, the limit in Assumption~\ref{assump} reduces to $E[\sigma(u(t,0))^2]$, and thus technical conditions such as those in Assumption~\ref{assump} are not required.

The nonlinear stochastic cable equation arises in the mathematical modeling of neurons, where it describes the propagation of electrical signals along their cylindrical structure.
Neurons, the fundamental components of the nervous system, operate through a sophisticated interplay of chemical, biological, and electrical phenomena.
In a common mathematical simplification, a neuron is idealized as a long, thin cylinder, similar to an electrical cable.
For such a cylinder on the interval $[0, L]$, we assume the electrical potential $u(t,x)$ depends only on position $x \in [0,L]$ and time $t$.
While this potential $u(t,x)$ is accurately governed by the Hodgkin--Huxley equations, for specific ranges of $u$, these equations can be closely approximated by the cable equation:
\begin{equation}\label{det_cableeq}
    \frac{\partial u}{\partial t} = \frac{\beta}{2} \frac{\partial^2 u}{\partial x^2} - \alpha u ,
\end{equation}
with $\beta / 2 > 0$ representing the diffusion rate within the neuron and $\alpha \in \mathbb{R}$ being the rate of ion leakage through the membrane~\cite{kallianpur_xiong}.
The neuron's surface receives current impulses through synapses.
If this incoming current at $(t,x)$ is represented by $F(t,x)$, the system obeys the inhomogeneous PDE:
\begin{equation}\nonumber
    \frac{\partial u}{\partial t} = \frac{\beta}{2} \frac{\partial^2 u}{\partial x^2} - \alpha u + F .
\end{equation}
Even in a resting state, occasional random impulses may occur, which implies that $F$ generally contains a stochastic component.
For instance, Walsh studied the case where $F$ is a compound Poisson process or a space-time white noise~\cite{walsh_2}.
The nonlinear stochastic cable equation with $F = \sigma(u) \dot{W}$ was analyzed in~\cite{walsh}.

\section{Preliminaries}
For any $Z \in L^k(\Omega)$, we define $\| Z \|_k := E[|Z|^k]^{1/k}$.

Throughout this paper, we assume that $\sigma$ is globally Lipschitz continuous with constant $K_{\sigma}$.
This implies the following linear growth condition:
\begin{equation}\nonumber
    |\sigma(y)| \leq M_{\sigma}(1 + |y|)
\end{equation}
where $M_{\sigma} = \max \{ K_{\sigma}, |\sigma(0)| \}$.

\subsection{Clark--Ocone formula}
Let $\mathcal{H}$ be the Hilbert space $L^2([0,T] \times \mathbb{R})$. 
The family of stochastic integrals
\begin{equation*}
    X(h) = \int_{[0,T] \times \mathbb{R}} h(s,x) \ W(\mathrm{d}s, \mathrm{d} x)
\end{equation*}
constitutes an isonormal Gaussian process $\{ X(h) \}_{h \in \mathcal{H}}$.
In this context, we employ the tools of Malliavin calculus (see, e.g.,~\cite{nualart}), and let $D$ denote the associated Malliavin derivative operator.
Let $\{\mathcal{F}_s\}_{s \in [0,T]}$ be the filtration generated by the space-time white noise $\dot{W}$.
A key result is the Clark--Ocone formula, which provides the following representation for any $\mathcal{F}_T$-measurable random variable $F$ in the Sobolev space $\mathbb{D}^{1,2}$:
\begin{equation*}
    F = E[F] + \int_{[0,T] \times \mathbb{R}} E[ D_{s,z} F \mid \mathcal{F}_s ] \ W(\mathrm{d}s, \mathrm{d} z) \quad \text{a.s.}
\end{equation*}
A well-known consequence of this formula, obtained via Jensen's inequality for conditional expectations, is the following Poincaré-type inequality:
\begin{equation}\label{poincare}
    \left| \operatorname{Cov}[F,G] \right| \leq \int_0^T \int_{\mathbb{R}} \| D_{s,z} F \|_2 \| D_{s,z} G \|_2 \ \mathrm{d} z \mathrm{d} s,
\end{equation}
which holds for any pair of $\mathcal{F}_T$-measurable random variables $F, G \in \mathbb{D}^{1,2}$.

\subsection{The Malliavin--Stein method}
We now recall some key results from the Malliavin--Stein method.
This method merges the Malliavin calculus with Stein's method to provide quantitative bounds on the distance between probability measures, proving particularly effective for central limit theorems~\cite{nourdin_peccati}.
\begin{proposition}
    \label{malliavin-stein1}
    Let $F = \delta(v)$ for some $v \in \dom(\delta)$, where $\dom(\delta)$ is the $L^2(\Omega)$-domain of the adjoint of the Malliavin derivative operator.
    Suppose that $F \in \mathbb{D}^{1,2}$ and $E[F^2] = 1$.
    Then, the following inequality holds:
    \begin{equation}\nonumber
        d_{\mathrm{TV}}(F, \mathcal{N}(0,1)) \leq 2 \sqrt{ \var( \langle D F, v \rangle_{\mathcal{H}} ) } .
    \end{equation}
\end{proposition}
The proof of the functional CLT in Theorem~\ref{thm2} requires a multivariate version of Proposition~\ref{malliavin-stein1}, which we state below.

\begin{proposition}
    \label{malliavin-stein2}
    Fix an integer $m \geq 2$.
    Let $F = (F_1, \dots, F_m)$ be a random vector where each component has the form $F_i = \delta(v_i)$ for some $v_i \in \dom(\delta)$ and $F_i \in \mathbb{D}^{1,2}$.
    Let $Z$ denote an $m$-dimensional centered Gaussian random vector with a given covariance matrix $(C_{i,j})_{1 \leq i,j \leq m}$.
    Then, for any function $h \in C^2(\mathbb{R}^m)$ with bounded second partial derivatives, we have
    \begin{equation}\nonumber
        \left| E[h(F)] - E[h(Z)] \right| \leq \frac{1}{2} \| h'' \|_{\infty} \sqrt{ \sum_{i,j=1}^m E\left[ \left| C_{i,j} - \langle D F_i, v_j \rangle_{\mathcal{H}} \right|^2 \right] } ,
    \end{equation}
    where
    \begin{equation}\nonumber
        \| h'' \|_{\infty} := \max_{1 \leq i,j \leq m} \sup_{x \in \mathbb{R}^m} \left| \frac{ \partial^2 h(x) }{ \partial x_i \partial x_j } \right| .
    \end{equation}
\end{proposition}

\subsection{Moments and Malliavin derivative of the solution}
Recently, Pu proved that the moments of $u(t,x)$ are uniformly bounded for all $L \geq 1$ in the case where $\alpha = 0$ and $\beta = 1$~\cite{pu}*{Lemma 2.3.}.
Similarly, a corresponding result has been obtained for the case $\alpha \in \mathbb{R}$ and $\beta > 0$.

\begin{lemma}
    Fix $T > 0$.
    Then for all $k \geq 2$, there exists $c_{T,k} > 0$ such that
    \begin{equation}\label{momentup}
        \sup_{L \geq 1} \sup_{(t,x) \in [0,T] \times [0,L]} \| u(t,x) \|_k \leq c_{T,k} < \infty .
    \end{equation}
\end{lemma}

Pu also proved that the moments of the Malliavin derivative of $u(t,x)$ satisfy a Gaussian-type upper bound uniformly over $L \geq 1$ in the case where $\alpha = 0$ and $\beta = 1$~\cite{pu}*{Lemma 2.4.}.
This result can similarly be extended to the case $\alpha \in \mathbb{R}$ and $\beta > 0$.

\begin{lemma}
    \label{mdup}
    Fix $T > 0$.
    Then, for all $k \geq 2$, there exists $C_{T,k} > 0$ such that for all $L \geq 1$, $(t,x) \in [0,T] \times [0,L]$, and for almost every $(s,y) \in (0,t) \times \mathbb{R}$,
    \begin{equation}\nonumber
        \| D_{s,y}u(t,x) \|_k \leq
        \begin{cases}
            C_{T,k} 1_{[0,L]}(y) p_{\beta(t-s)}(x-y) & \mathrm{(Neumann/Dirichlet ~ case)} \\
            C_{T,k} 1_{[0,L]}(y) G_{t-s}(x,y) & \mathrm{(periodic ~ case)}
        \end{cases}
    \end{equation}
    where $p_t(z)$ denotes the Gaussian heat kernel, defined in \eqref{gaussian_heat_kernel}.
    In particular, we have $u(t,x) \in \bigcap_{k \geq 2} \mathbb{D}^{1,k}$ for all $(t,x) \in [0,T] \times [0,L]$.
\end{lemma}

\section{Asymptotic behavior of the covariance}
In this section, we analyze the asymptotic behavior of the covariance of the spatial integral of the solution to \eqref{cableeq}.

Since \eqref{Gg}, by the same arguments as~\cite{pu}*{Lemma 3.2. and Lemma A.4.}, we obtain the following supporting lemma.
\begin{lemma}
    \label{I0lem}
    Fix $T > 0$.
    Denote for $(t,x) \in (0,T] \times [0,L]$,
    \begin{equation}\nonumber
        \mathcal{I}_0(t,x) = \int_0^L G_t(x,y) \ \mathrm{d} y , \text{~and} \quad \mathcal{I}_0(0,x) = 1 .
    \end{equation}
    Then
    \begin{equation}\nonumber
        \sup_{L \geq 1} \sup_{(t,x) \in [0,T] \times [0,L]} \mathcal{I}_0(t,x) \leq e^{|\alpha| T} ,
    \end{equation}
    and for all $t > 0$,
    \begin{equation}\nonumber
        \lim_{L \to \infty} \frac{1}{L} \int_0^L \mathcal{I}_0(t,x) \ \mathrm{d} x = e^{-\alpha t} ,
    \end{equation}
    and for all $t_1 , t_2 > 0$,
    \begin{equation}\nonumber
        \lim_{L \to \infty} \frac{1}{L} \int_0^L \mathcal{I}_0(t_1,x) \mathcal{I}_0(t_2, x) \ \mathrm{d} x = e^{-\alpha(t_1 + t_2)} .
    \end{equation}
\end{lemma}

The following lemmas are prepared for the proof of Proposition \ref{asymp_cov_low}.
\begin{lemma}
    \label{lem_neumann_periodic}
    In the case of Neumann or periodic boundary conditions,
    \begin{equation}\nonumber
        \lim_{t \to 0} \sup_{L \geq 1} \sup_{x \in [0,L]} | E[\sigma(u(t,x))^2] - \sigma(1)^2 | = 0 .
    \end{equation}
\end{lemma}

\begin{proof}
    Fix $T > 0$.
    For every $(t,x) \in [0,T] \times [0,L]$,
    \begin{equation}\nonumber
        \begin{aligned}
            &| E[\sigma(u(t,x))^2] - \sigma(1)^2 | \\
            &= | E[ ( \sigma(u(t,x)) - \sigma(1) ) ( \sigma(u(t,x)) + \sigma(1) ) ] | \\
            &\leq \sqrt{ E[ | \sigma(u(t,x)) - \sigma(1) |^2 ] } \sqrt{ E[ | \sigma(u(t,x)) + \sigma(1) |^2 ] } \\
            &\leq \sqrt{E[ K_{\sigma}^2 | u(t,x) - 1 |^2  ]} \sqrt{ E[  2\sigma(u(t,x))^2 + 2\sigma(1)^2  ] } \\
            &\leq K_{\sigma} \sqrt{ E[ | u(t,x) - 1 |^2 ] } \sqrt{ 4 M_{\sigma}^2 (c_{T,2}^2 + 1) + 2 \sigma(1)^2 } ,
        \end{aligned}
    \end{equation}
    where $K_{\sigma}$ is the Lipschitz constant of $\sigma$, and $M_{\sigma}$ is the constant in the linear growth condition.
    From \eqref{mildsol} and It\^o's isometry, we get
    \begin{equation}\label{mildsol_ineq}
        \begin{aligned}
            &E[ | u(t,x) - 1 |^2 ] \\
            &\leq 2 \left| \int_0^L G_t(x,y) \ \mathrm{d} y - 1 \right|^2 + 2 \int_0^t \int_0^L G_{t-s}(x,y)^2 E[ \sigma^2(u(s,y)) ] \ \mathrm{d} s \mathrm{d} y \\
            &\leq 2 \left| \int_0^L G_t(x,y) \ \mathrm{d} y - 1 \right|^2 + 4 M_{\sigma}^2 (1 + c_{T,2}^2) \int_0^t \int_0^L G_{t-s}(x,y)^2 \ \mathrm{d} s \mathrm{d} y .
        \end{aligned}
    \end{equation}
    Since we assume Neumann or periodic boundary conditions, using \eqref{gid} and semigroup property of $G$, we have
    \begin{equation}\nonumber
        \begin{aligned}
            &E[ | u(t,x) - 1 |^2 ] \\
            &\leq 2 |e^{-\alpha t} - 1|^2 + 4 M_{\sigma}^2 (1 + c_{T,2}^2) \int_0^t G_{2(t-s)}(x,x) \ \mathrm{d} s 
        \end{aligned}
    \end{equation}
    In the case of Neumann boundary conditions, using \eqref{const_kt}, we get
    \begin{equation}\nonumber
        \begin{aligned}
            \int_0^t G_{2(t-s)}(x,x) \ \mathrm{d} s &\leq K_{2T} e^{2|\alpha|T}  \int_0^t \frac{1}{\sqrt{4 \pi \beta(t-s)}} \ \mathrm{d} s \\
            &= K_{2T} e^{2|\alpha|T} \sqrt{\frac{t}{\pi \beta}} .
        \end{aligned}
    \end{equation}
    Then, we have
    \begin{equation}\nonumber
        \begin{aligned}
            &E[ | u(t,x) - 1 |^2 ] \\
            &\leq 2 |e^{-\alpha t} - 1|^2 + 4 M_{\sigma}^2 (1 + c_{T,2}^2) K_{2T} e^{2|\alpha|T} \sqrt{\frac{t}{\pi \beta}}
        \end{aligned}
    \end{equation}
    In the case of periodic boundary conditions, using \eqref{periodic_rep} and the following identity~\cite{pu}*{A.17}
    \begin{equation}\nonumber
        \sum_{j \in \mathbb{Z}} p_{r}(j) = \sum_{n \in \mathbb{Z}} e^{-2r \pi^2 n^2} \quad \text{for}~\text{all}~ r > 0 ,
    \end{equation}
    we have
    \begin{equation}\nonumber
        \begin{aligned}
            &\int_0^t G_{2(t-s)}(x,x) \ \mathrm{d} s \\
            &= \int_0^t e^{-2\alpha (t-s)} \sum_{j \in \mathbb{Z}} p_{2\beta (t-s)}(jL) \ \mathrm{d} s \\
            &\leq \int_0^t e^{-2\alpha (t-s)} \sum_{j \in \mathbb{Z}} p_{2\beta (t-s)}(j) \ \mathrm{d} s \\
            &= \int_0^t e^{-2\alpha (t-s)} \sum_{n \in \mathbb{Z}} e^{-4\beta (t-s) \pi^2 n^2}\ \mathrm{d} s \\
            &\leq e^{2|\alpha|T} \sum_{n \in \mathbb{Z}} \int_0^t e^{-4\beta (t-s) \pi^2 n^2}\ \mathrm{d} s \\
            &= e^{2|\alpha|T} \sum_{n \in \mathbb{Z}} \frac{1 - e^{-4\beta t \pi^2 n^2}}{ 4\beta \pi^2 n^2 } .
        \end{aligned}
    \end{equation}
    Then,
    \begin{equation}\nonumber
        \begin{aligned}
            &E[ | u(t,x) - 1 |^2 ] \\
            &\leq 2 |e^{-\alpha t} - 1|^2 + 4 M_{\sigma}^2 (1 + c_{T,2}^2) e^{2|\alpha|T} \sum_{n \in \mathbb{Z}} \frac{1 - e^{-4\beta t \pi^2 n^2}}{ 4\beta \pi^2 n^2 } .
        \end{aligned}
    \end{equation}
    Applying the Lebesgue dominated convergence theorem, we get
    \begin{equation}\nonumber
        \lim_{t \to 0} \sum_{n \in \mathbb{Z}} \frac{1 - e^{-4\beta t \pi^2 n^2}}{ 4\beta \pi^2 n^2 } = \sum_{n \in \mathbb{Z}} \lim_{t \to 0} \frac{1 - e^{-4\beta t \pi^2 n^2}}{ 4\beta \pi^2 n^2 } = 0 .
    \end{equation}
    Therefore, in the case of Neumann or periodic boundary conditions,
    \begin{equation}\nonumber
        \lim_{t \to 0} \sup_{L \geq 1} \sup_{x \in [0,L]} | E[\sigma(u(t,x))^2] - E[\sigma(1)^2] | = 0 .
    \end{equation}
\end{proof}

\begin{lemma}
    \label{lem_dirichlet}
    In the case of Dirichlet boundary conditions,
    \begin{equation}\nonumber
        \lim_{t \to 0} \sup_{L \geq 1} \left| \frac{1}{L} \int_0^L E[\sigma(u(t,x))^2] \ \mathrm{d} x - \sigma(1)^2 \right| = 0 .
    \end{equation}
\end{lemma}

\begin{proof}
    Fix $T > 0$.
    For every $(t,x) \in [0,T] \times [0,L]$,
    \begin{equation}\nonumber
        \begin{aligned}
            &\left| \frac{1}{L} \int_0^L E[\sigma(u(t,x))^2] \ \mathrm{d} x - \sigma(1)^2 \right| \\
            &= \left| \frac{1}{L} \int_0^L E[ ( \sigma(u(t,x)) - \sigma(1) ) ( \sigma(u(t,x)) + \sigma(1) ) ] \ \mathrm{d} x \right| \\
            &\leq \sqrt{ \frac{1}{L} \int_0^L E[ | \sigma(u(t,x)) - \sigma(1) |^2 ] \ \mathrm{d} x } \sqrt{ \frac{1}{L} \int_0^L E[ | \sigma(u(t,x)) + \sigma(1) |^2 ] \ \mathrm{d} x } \\
            &\leq \sqrt{ \frac{1}{L} \int_0^L E[ K_{\sigma}^2 | u(t,x) - 1 |^2 ] \ \mathrm{d} x } \sqrt{ \frac{1}{L} \int_0^L E[  2\sigma(u(t,x))^2 + 2\sigma(1)^2 ] \ \mathrm{d} x } \\
            &\leq K_{\sigma} \sqrt{ \frac{1}{L} \int_0^L E[ | u(t,x) - 1 |^2 ] \ \mathrm{d} x } \sqrt{ 4 M_{\sigma}^2 (c_{T,2}^2 + 1) + 2 \sigma(1)^2 } .
        \end{aligned}
    \end{equation}
    From \eqref{mildsol} and It\^o's isometry, we get \eqref{mildsol_ineq}.
    Since we assume Dirichlet boundary conditions, using semigroup property of $G$ and \eqref{const_kt}, we have
    \begin{equation}\nonumber
        \begin{aligned}
            &E[ | u(t,x) - 1 |^2 ] \\
            &\leq 2 \left| \int_0^L G_t(x,y) \ \mathrm{d} y - 1 \right|^2 + 4 M_{\sigma}^2 (1 + c_{T,2}^2) \int_0^t G_{2(t-s)}(x,x) \ \mathrm{d} s \\
            &\leq 2 \left| \int_0^L G_t(x,y) \ \mathrm{d} y - 1 \right|^2 + 4 M_{\sigma}^2 (1 + c_{T,2}^2) K_{2T} e^{2|\alpha|T} \int_0^t \frac{1}{\sqrt{4 \pi \beta(t-s)}} \ \mathrm{d} s \\
            &\leq 2 \left| \int_0^L G_t(x,y) \ \mathrm{d} y - 1 \right|^2 + 4 M_{\sigma}^2 (1 + c_{T,2}^2) K_{2T} e^{2|\alpha|T} \sqrt{\frac{t}{\pi \beta}} .
        \end{aligned}
    \end{equation}
    By \eqref{dirichlet_rep}, we get
    \begin{equation}\nonumber
        \begin{aligned}
            &\int_0^L G_t(x,y) \ \mathrm{d} y \\
            &= 2 e^{-\alpha t} \sum_{n=1}^{\infty} \sin \left( \frac{n \pi x}{L} \right) \frac{1 - (-1)^n}{n \pi } e^{-\frac{n^2 \pi^2 \beta t}{2 L^2}} .
        \end{aligned}
    \end{equation}
    Then, applying the $L^2([0,L])$-orthogonality of the functions $\{ x \mapsto \sin(n \pi x / L) \}_{n=1}^{\infty}$,
    \begin{equation}\label{int_gt_eq}
        \begin{aligned}
            &\frac{1}{L} \int_0^L \left| \int_0^L G_t(x,y) \ \mathrm{d} y - 1 \right|^2 \ \mathrm{d} x \\
            &= 4 e^{-2\alpha t} \sum_{n=1}^{\infty} \frac{1}{L} \int_0^L \sin^2 \left( \frac{n \pi x}{L} \right) \ \mathrm{d} x \left( \frac{1 - (-1)^n}{n \pi} \right)^2 e^{-\frac{n^2 \pi^2 \beta t}{L^2}} \\
            &\quad + 1 - 4 e^{-\alpha t} \sum_{n=1}^{\infty} \frac{1}{L} \int_0^L \sin \left( \frac{n \pi x}{L} \right) \ \mathrm{d} x  \frac{1 - (-1)^n}{n \pi} e^{-\frac{n^2 \pi^2 \beta t}{2 L^2}} \\
            &= 2 e^{-2\alpha t} \sum_{n=1}^{\infty} \left( \frac{1 - (-1)^n}{n \pi} \right)^2 e^{-\frac{n^2 \pi^2 \beta t}{L^2}} + 1 - 4 e^{-\alpha t} \sum_{n=1}^{\infty} \left( \frac{1 - (-1)^n}{n \pi} \right)^2 e^{-\frac{n^2 \pi^2 \beta t}{2 L^2}} .
        \end{aligned}
    \end{equation}
    Now, for every $L \geq 1$,
    \begin{equation}\nonumber
        \begin{aligned}
            &\left| \sum_{n=1}^{\infty} \left( \frac{1 - (-1)^n}{n \pi} \right)^2 e^{-\frac{n^2 \pi^2 \beta t}{L^2}} - \sum_{n=1}^{\infty} \left( \frac{1 - (-1)^n}{n \pi} \right)^2 \right| \\
            &= \left| \sum_{n=1}^{\infty} \left( \frac{1 - (-1)^n}{n \pi} \right)^2 ( e^{-\frac{n^2 \pi^2 \beta t}{L^2}} - 1 ) \right| \\
            &\leq \sum_{n=1}^{\infty} \left( \frac{1 - (-1)^n}{n \pi} \right)^2 | e^{-\frac{n^2 \pi^2 \beta t}{L^2}} - 1 | \\
            &\leq \sum_{n=1}^{\infty} \left( \frac{1 - (-1)^n}{n \pi} \right)^2 | e^{- n^2 \pi^2 \beta t} - 1 | .
        \end{aligned}
    \end{equation}
    Applying the dominated convergence theorem,
    \begin{equation}\nonumber
        \lim_{t \to 0} \sum_{n=1}^{\infty} \left( \frac{1 - (-1)^n}{n \pi} \right)^2 | e^{- n^2 \pi^2 \beta t} - 1 | = 0.
    \end{equation}
    Combining the above with the following identity
    \begin{equation}\nonumber
        \sum_{n=1}^{\infty} \left( \frac{1 - (-1)^n}{n \pi} \right)^2 = \sum_{k=1}^{\infty} \frac{4}{(2k - 1)^2 \pi^2} = \frac{1}{2} 
    \end{equation}
    yields
    \begin{equation}\nonumber
        \lim_{t \to 0} \sup_{L \geq 1} \left| \sum_{n=1}^{\infty} \left( \frac{1 - (-1)^n}{n \pi} \right)^2 e^{-\frac{n^2 \pi^2 \beta t}{L^2}} - \frac{1}{2} \right| = 0 .
    \end{equation}
    Similarly, we get
    \begin{equation}\nonumber
        \lim_{t \to 0} \sup_{L \geq 1} \left| \sum_{n=1}^{\infty} \left( \frac{1 - (-1)^n}{n \pi} \right)^2 e^{-\frac{n^2 \pi^2 \beta t}{2L^2}} - \frac{1}{2} \right| = 0 .
    \end{equation}
    These together with \eqref{int_gt_eq} imply 
    \begin{equation}\nonumber
        \lim_{t \to 0} \sup_{L \geq 1} \frac{1}{L} \int_0^L \left| \int_0^L G_t(x,y) \ \mathrm{d} y - 1 \right|^2 \ \mathrm{d} x = 0 .
    \end{equation}
    Therefore, we have
    \begin{equation}\nonumber
        \lim_{t \to 0} \sup_{L \geq 1} \frac{1}{L} \int_0^L E[ | u(t,x) - 1 |^2 ] \ \mathrm{d} x = 0 .
    \end{equation}
    This completes the proof.

\end{proof}

The following results provide the asymptotic behavior of the covariance function of the renormalized sequence of processes $F_L(t)$ as $L$ tends to infinity.
\begin{proposition}
    \label{asymp_cov_low}
    There exists $\delta > 0$, for every $t_1, t_2 > 0$,
    \begin{equation}\label{eq_asymp_cov_low}
        \begin{aligned}
            &\liminf_{L \to \infty} \cov\left[ \sqrt{L} F_L(t_1), \sqrt{L} F_L(t_2) \right] \geq \frac{\sigma(1)^2}{2} \int_0^{t_1 \wedge t_2 \wedge \delta} e^{-\alpha(t_1 + t_2 - 2s)} \ \mathrm{d} s .
        \end{aligned}
    \end{equation}
\end{proposition}

\begin{proof}
    Using the mild form in \eqref{mildsol} and It\^o's isometry, we have
    \begin{equation}\label{eq_guarantees1}
        \begin{aligned}
            &\cov\left[ \sqrt{L} F_L(t_1), \sqrt{L} F_L(t_2) \right] \\
            &= \frac{1}{L} \int_{[0,L]^2} \cov \left[ u(t_1,x) , u(t_2,y) \right] \ \mathrm{d} x \mathrm{d} y \\
            &= \frac{1}{L} \int_{[0,L]^2} \int_0^{t_1 \wedge t_2} \int_0^L G_{t_1 - s}(x,z) G_{t_2 - s}(y,z) E[\sigma(u(s,z))^2] \ \mathrm{d} z \mathrm{d} s \mathrm{d} x \mathrm{d} y \\
            &= \frac{1}{L} \int_0^{t_1 \wedge t_2} \int_0^L \mathcal{I}_0(t_1 - s, z) \mathcal{I}_0(t_2 - s, z) E[\sigma(u(s,z))^2] \ \mathrm{d} z \mathrm{d} s ,
        \end{aligned}
    \end{equation}
    where $\mathcal{I}_0$ is defined in Lemma~\ref{I0lem}.
    Moreover, we obtain
    \begin{equation}\nonumber
        \begin{aligned}
            &\cov\left[ \sqrt{L} F_L(t_1), \sqrt{L} F_L(t_2) \right] \\
            &= \frac{1}{L} \int_0^{t_1 \wedge t_2} \int_0^L \left[ \mathcal{I}_0(t_1 - s, z) \mathcal{I}_0(t_2 - s, z) - e^{-\alpha(t_1 + t_2 - 2s)} \right] E[\sigma(u(s,z))^2] \ \mathrm{d} z \mathrm{d} s \\
            &\quad + \int_0^{t_1 \wedge t_2} e^{-\alpha(t_1 + t_2 - 2s)} \frac{1}{L} \int_0^L E[\sigma(u(s,z))^2] \ \mathrm{d} z \mathrm{d} s .
        \end{aligned}
    \end{equation}
    By \eqref{momentup},
    \begin{equation}\nonumber
        \begin{aligned}
            &\left| \frac{1}{L} \int_0^{t_1 \wedge t_2} \int_0^L \left[ \mathcal{I}_0(t_1 - s, z) \mathcal{I}_0(t_2 - s, z) - e^{-\alpha(t_1 + t_2 - 2s)} \right] E[\sigma(u(s,z))^2] \ \mathrm{d} z \mathrm{d} s \right| \\
            &\leq 2 M_{\sigma}^2 (1 + c_{T,2}^2) \int_0^{t_1 \wedge t_2} \left| \frac{1}{L} \int_0^L \mathcal{I}_0(t_1 - s, z) \mathcal{I}_0(t_2 - s, z) \ \mathrm{d} z - e^{-\alpha(t_1 + t_2 - 2s)}  \right| \ \mathrm{d} s ,
        \end{aligned}
    \end{equation}
    where $M_{\sigma} > 0$ is the constant in the linear growth condition for $\sigma$.
    Hence, applying Lemma~\ref{I0lem} and dominated convergence theorem, we obtain
    \begin{equation}\nonumber
        \begin{aligned}
            &\frac{1}{L} \int_0^{t_1 \wedge t_2} \int_0^L \left[ \mathcal{I}_0(t_1 - s, z) \mathcal{I}_0(t_2 - s, z) - e^{-\alpha(t_1 + t_2 - 2s)} \right] E[\sigma(u(s,z))^2] \ \mathrm{d} z \mathrm{d} s \\
            &\to 0 \mathrm{~as~} L \to \infty .
        \end{aligned}
    \end{equation}
    Then, we have
    \begin{equation}\label{eq_liminf}
        \begin{aligned}
            &\liminf_{L \to \infty} \cov\left[ \sqrt{L} F_L(t_1), \sqrt{L} F_L(t_2) \right] \\
            &= \liminf_{L \to \infty} \int_0^{t_1 \wedge t_2} e^{-\alpha(t_1 + t_2 - 2s)} \frac{1}{L} \int_0^L E[\sigma(u(s,z))^2] \ \mathrm{d} z \mathrm{d} s \\ .
        \end{aligned}
    \end{equation}
    In the case of Neumann or periodic boundary conditions, from Lemma \ref{lem_neumann_periodic}, there exists $\delta > 0$, for every $L \geq 1$ and $(s,z) \in [0, \delta] \times [0,L]$,
    \begin{equation}\label{eq_guarantees2}
        E[\sigma(u(s,z))^2] \geq \frac{\sigma(1)^2}{2} .
    \end{equation}
    In the case of Dirichlet boundary conditions, from Lemma \ref{lem_dirichlet}, there exists $\delta > 0$, for every $L \geq 1$ and $s \in [0, \delta]$,
    \begin{equation}\label{eq_guarantees3}
        \frac{1}{L} \int_0^L E[\sigma(u(s,z))^2] \ \mathrm{d} z \geq \frac{\sigma(1)^2}{2} .
    \end{equation}
    These together with \eqref{eq_liminf} implies \eqref{eq_asymp_cov_low}.
\end{proof}

\begin{proposition}
    \label{asymp_cov}
    Suppose that Assumption \ref{assump} holds.
    For every $t_1, t_2 > 0$,
    \begin{equation}\label{asymp_cov_eq}
        \lim_{L \to \infty} \cov\left[ \sqrt{L} F_L(t_1), \sqrt{L} F_L(t_2) \right] = \int_0^{t_1 \wedge t_2} e^{-\alpha(t_1 + t_2 - 2s)} f_{\sigma}(s) \ \mathrm{d} s ,
    \end{equation}
    where the function $f_{\sigma}$ is defined in Assumption~\ref{assump}.
\end{proposition}

\begin{proof}
    By the same argument as in the proof of Proposition \ref{asymp_cov_low}, for every $t_1, t_2 > 0$,
    \begin{equation}\nonumber
        \begin{aligned}
            &\lim_{L \to \infty} \cov\left[ \sqrt{L} F_L(t_1), \sqrt{L} F_L(t_2) \right] \\
            &= \lim_{L \to \infty} \int_0^{t_1 \wedge t_2} e^{-\alpha(t_1 + t_2 - 2s)} \frac{1}{L} \int_0^L E[\sigma(u(s,z))^2] \ \mathrm{d} z \mathrm{d} s 
        \end{aligned}
    \end{equation}
    This together with Assumption~\ref{assump} implies \eqref{asymp_cov_eq}.
\end{proof}

\section{Proof of Theorem~\ref{thm1}}
Using stochastic Fubini's theorem, we have
\begin{equation}\label{flt2}
    F_L(t) = \int_0^t \int_0^L v_{L,t}(s,y) \ W(\mathrm{d}s, \mathrm{d} y) = \delta(v_{L,t}) \mathrm{~a.s.} ,
\end{equation}
where
\begin{equation}\label{vlt}
    \begin{aligned}
        &v_{L,t}(s,y) \\
        &:= \frac{1}{L} 1_{(0,t)}(s) 1_{[0,L]}(y) \sigma(u(s,y)) \int_0^L G_{t-s}(x,y) \ \mathrm{d} x \\
        &= \frac{1}{L} 1_{(0,t)}(s) 1_{[0,L]}(y) \sigma(u(s,y)) \mathcal{I}_0(t-s,y) ,
    \end{aligned}
\end{equation}
and $\delta$ is the adjoint of the Malliavin derivative operator.

We prepare the following proposition for the proof of our main results.
\begin{proposition}
    \label{keyprop}
    For every $T > 0$, there exists $A_T > 0$ such that
    \begin{equation}\nonumber
        \sup_{t,\tau \in [0,T]} \var \left( \left\langle_{} D F_L(t), v_{L,\tau} \right\rangle_{\mathcal{H}} \right) \leq \frac{A_T}{L^3} \text{~for~all~} L \geq 1.
    \end{equation}
\end{proposition}

\begin{proof}
    From Proposition $1.3.2$ of~\cite{nualart} and \eqref{flt2}, we have
    \begin{equation}\nonumber
        D_{r,z} F_L(t) = 1_{(0,t)}(r) v_{L,t}(r,z) + 1_{(0,t)}(r) \int_r^t \int_0^L D_{r,z} v_{L,t}(s,y) \ W(\mathrm{d} s, \mathrm{d} y) .
    \end{equation}
    Hence, using the stochastic Fubini's theorem, we obtain
    \begin{equation}\nonumber
        \begin{aligned}
            &\left\langle_{} D F_L(t), v_{L,\tau} \right\rangle_{\mathcal{H}} \\
            &= \left\langle_{} v_{L, t} , v_{L,\tau} \right\rangle_{\mathcal{H}} + \int_0^{\tau} \int_0^L v_{L,\tau}(r,z) \left( \int_r^t \int_0^L D_{r,z} v_{L,t}(s,y) \ W(\mathrm{d}s, \mathrm{d}y) \right) \mathrm{d} z \mathrm{d} r \\
            &= \left\langle_{} v_{L, t} , v_{L,\tau} \right\rangle_{\mathcal{H}} + \int_0^t \int_0^L \left( \int_0^{\tau \wedge s} \int_0^L v_{L,\tau}(r,z) D_{r,z} v_{L,t}(s,y) \ \mathrm{d} z \mathrm{d} r \right) W(\mathrm{d}s, \mathrm{d}y) .
        \end{aligned}
    \end{equation}
    Therefore,
    \begin{equation}\nonumber
        \var \left( \left\langle_{} D F_L(t), v_{L,\tau} \right\rangle_{\mathcal{H}} \right) \leq 2( \Phi_{L,t,\tau}^{(1)} + \Phi_{L,t,\tau}^{(2)} ) ,
    \end{equation}
    where
    \begin{equation}\nonumber
        \Phi_{L,t,\tau}^{(1)} = \var \left( \left\langle_{} v_{L,t}, v_{L,\tau} \right\rangle_{\mathcal{H}} \right) ,
    \end{equation}
    \begin{equation}\nonumber
        \Phi_{L,t,\tau}^{(2)} = \var \left( \int_0^t \int_0^L \left( \int_0^{\tau \wedge s} \int_0^L v_{L,\tau}(r,z) D_{r,z} v_{L,t}(s,y) \ \mathrm{d} z \mathrm{d} r \right) W(\mathrm{d}s, \mathrm{d}y) \right) .
    \end{equation}
    From \eqref{vlt},
    \begin{equation}\nonumber
        \begin{aligned}
            &\Phi_{L,t,\tau}^{(1)} \\
            &= \frac{1}{L^4} \var \left( \int_0^{t \wedge \tau} \int_0^L \sigma(u(s,y))^2 \mathcal{I}_0(t-s,y) \mathcal{I}_0(\tau -s,y) \ \mathrm{d} y \mathrm{d} s \right) \\
            &= \frac{1}{L^4} \int_{[0,t \wedge \tau]^2} \int_{[0,L]^2} \cov \left[ \sigma(u(s_1,y_1))^2 , \sigma(u(s_2,y_2))^2 \right] \\
            &\quad \times \mathcal{I}_0(t-s_1,y_1) \mathcal{I}_0(\tau -s_1,y_1) \mathcal{I}_0(t-s_2,y_2) \mathcal{I}_0(\tau -s_2,y_2) \ \mathrm{d} y_1 \mathrm{d} y_2 \mathrm{d} s_1 \mathrm{d} s_2 .
        \end{aligned}
    \end{equation}
    By Lemma~\ref{I0lem}, H\"older's inequality and Minkowski inequality, we have
    \begin{equation}\nonumber
        \begin{aligned}
            &\Phi_{L,t,\tau}^{(1)} \\
            &\leq \frac{4 e^{4 |\alpha| T}}{L^4} \int_{[0,t \wedge \tau]^2} \int_{[0,L]^2} \int_0^{s_1 \wedge s_2} \int_{\mathbb{R}} \| \sigma(u(s_1,y_1)) D_{r,z} \sigma(u(s_1,y_1)) \|_2 \\
            &\quad \times \| \sigma(u(s_2,y_2)) D_{r,z} \sigma(u(s_2,y_2)) \|_2 \ \mathrm{d} z \mathrm{d} r \mathrm{d} y_1 \mathrm{d} y_2 \mathrm{d} s_1 \mathrm{d} s_2 \\
            &\leq \frac{4 e^{4 |\alpha| T} M_{\sigma}^2 (1 + c_{T,4})^2 }{L^4} \int_{[0,t \wedge \tau]^2} \int_{[0,L]^2} \int_0^{s_1 \wedge s_2} \int_{\mathbb{R}} \| D_{r,z} \sigma(u(s_1,y_1)) \|_4 \| D_{r,z} \sigma(u(s_2,y_2)) \|_4 \\
            &\quad \times \mathrm{d} z \mathrm{d} r \mathrm{d} y_1 \mathrm{d} y_2 \mathrm{d} s_1 \mathrm{d} s_2 ,
        \end{aligned}
    \end{equation}
    where $M_{\sigma}$ is the constant in the linear growth condition for $\sigma$.
    From the chain rule of Malliavin derivative for a Lipschitz function with constant $K_{\sigma}$~\cite{nualart},
    \begin{equation}\nonumber
        D_{r,z} \sigma(u(s,y)) = G_{\sigma , u(s,y)} D_{r,z} u(s,y) ,
    \end{equation}
    where $G_{\sigma , u(s,y)}$ is a random variable bounded by $K_{\sigma}$.
    Then, we obtain
    \begin{equation}\nonumber
        \begin{aligned}
            &\Phi_{L,t,\tau}^{(1)} \\
            &\leq \frac{4 e^{4 |\alpha| T} M_{\sigma}^2 (1 + c_{T,4})^2 K_{\sigma}^2 }{L^4} \int_{[0,t \wedge \tau]^2} \int_{[0,L]^2} \int_0^{s_1 \wedge s_2} \int_{\mathbb{R}} \| D_{r,z} u(s_1,y_1) \|_4 \| D_{r,z} u(s_2,y_2) \|_4 \\
            &\quad \times \mathrm{d} z \mathrm{d} r \mathrm{d} y_1 \mathrm{d} y_2 \mathrm{d} s_1 \mathrm{d} s_2 ,
        \end{aligned}
    \end{equation}
    where $K_{\sigma}$ is Lipschitz constant of $\sigma$.
    In the case of Neumann/Dirichlet boundary conditions, by Lemma~\ref{mdup}, we have
    \begin{equation}\nonumber
        \begin{aligned}
            &\Phi_{L,t,\tau}^{(1)} \\
            &\leq \frac{4 e^{4 |\alpha| T} M_{\sigma}^2 (1 + c_{T,4})^2 K_{\sigma}^2 C_{T,4}^2 }{L^4} \int_{[0,t \wedge \tau]^2} \int_{[0,L]^2} \int_0^{s_1 \wedge s_2} \int_{\mathbb{R}} p_{\beta(s_1 - r)}(y_1 - z) p_{\beta(s_2 - r)}(y_2 - z) \\
            &\quad \times \mathrm{d} z \mathrm{d} r \mathrm{d} y_1 \mathrm{d} y_2 \mathrm{d} s_1 \mathrm{d} s_2 \\
            &= \frac{4 e^{4 |\alpha| T} M_{\sigma}^2 (1 + c_{T,4})^2 K_{\sigma}^2 C_{T,4}^2 }{L^4} \int_{[0,t \wedge \tau]^2} \int_{[0,L]^2} \int_0^{s_1 \wedge s_2} p_{\beta(s_1 + s_2 - 2r)}(y_1 - y_2) \\
            &\quad \times \mathrm{d} z \mathrm{d} r \mathrm{d} y_1 \mathrm{d} y_2 \mathrm{d} s_1 \mathrm{d} s_2 ,
        \end{aligned}
    \end{equation}
    where the equality holds by the semigroup property of Gaussian heat kernel $p_t(z)$.
    By using the identity
    \begin{equation}\label{ptid}
        \int_{\mathbb{R}} p_{\beta(s_1 + s_2 - 2r)}(y_1) \mathrm{d} y_1 = 1 ,
    \end{equation}
    we obtain
    \begin{equation}\nonumber
        \begin{aligned}
            &\Phi_{L,t,\tau}^{(1)} \\
            &\leq \frac{4 e^{4 |\alpha| T} M_{\sigma}^2 (1 + c_{T,4})^2 K_{\sigma}^2 C_{T,4}^2 }{L^4} \int_{[0,t \wedge \tau]^2} \int_{[0,L]} \int_{\mathbb{R}}  \int_0^{s_1 \wedge s_2} p_{\beta(s_1 + s_2 - 2r)}(y_1) \\
            &\quad \times \mathrm{d} z \mathrm{d} r \mathrm{d} y_1 \mathrm{d} y_2 \mathrm{d} s_1 \mathrm{d} s_2 \\
            &= \frac{4 e^{4 |\alpha| T} M_{\sigma}^2 (1 + c_{T,4})^2 K_{\sigma}^2 C_{T,4}^2 }{L^3} \int_{[0,t \wedge \tau]^2} \int_0^{s_1 \wedge s_2} \ \mathrm{d} r \mathrm{d} s_1 \mathrm{d} s_2 \\
            &\leq \frac{4 T^3 e^{4 |\alpha| T} M_{\sigma}^2 (1 + c_{T,4})^2 K_{\sigma}^2 C_{T,4}^2 }{L^3} .
        \end{aligned}
    \end{equation}
    In the case of periodic boundary conditions, from Lemma~\ref{mdup}, we have
    \begin{equation}\nonumber
        \begin{aligned}
            &\Phi_{L,t,\tau}^{(1)} \\
            &\leq \frac{4 e^{4 |\alpha| T} M_{\sigma}^2 (1 + c_{T,4})^2 K_{\sigma}^2 C_{T,4}^2 }{L^4} \int_{[0,t \wedge \tau]^2} \int_{[0,L]^2} \int_0^{s_1 \wedge s_2} \int_0^L G_{s_1 - r}(y_1, z) G_{s_2 - r}(y_2, r) \\
            &\quad \times \mathrm{d} z \mathrm{d} r \mathrm{d} y_1 \mathrm{d} y_2 \mathrm{d} s_1 \mathrm{d} s_2 \\
            &= \frac{4 e^{4 |\alpha| T} M_{\sigma}^2 (1 + c_{T,4})^2 K_{\sigma}^2 C_{T,4}^2 }{L^3} \int_{[0,t \wedge \tau]^2} \int_0^{s_1 \wedge s_2} e^{-\alpha(s_1 + s_2 - r)} \mathrm{d} r \mathrm{d} s_1 \mathrm{d} s_2 \\
            &\leq \frac{4 T^3 e^{6 |\alpha| T} M_{\sigma}^2 (1 + c_{T,4})^2 K_{\sigma}^2 C_{T,4}^2 }{L^3} ,
        \end{aligned}
    \end{equation}
    where the equality follows from \eqref{gid}.
    
    We proceed to estimate $\Phi_{L,t,\tau}^{(2)}$.
    By It\^o's isometry, we have
    \begin{equation}\nonumber
        \begin{aligned}
            &\Phi_{L,t,\tau}^{(2)} \\
            &= \int_0^t \int_0^L \left\| \int_0^{\tau \wedge s} \int_0^L v_{L,\tau}(r,z) D_{r,z} v_{L,t}(s,y) \right\|_2^2 \ \mathrm{d} y \mathrm{d} s \\
            &= \frac{1}{L^4} \int_0^t \int_0^L \int_{[0,\tau \wedge s]^2} \int_{[0,L]^2} \mathcal{I}_0(\tau - r_1, z_1) \mathcal{I}_0(\tau - r_2, z_2) \mathcal{I}_0^2(t - s, y) \\
            &\quad \times E[ \sigma(u(r_1,z_1)) ( D_{r_1,z_1} \sigma(u(s,y)) ) \sigma(u(r_2,z_2)) ( D_{r_2,z_2} \sigma(u(s,y)) ) ] \ \mathrm{d} z_1 \mathrm{d} z_2 \mathrm{d} r_1 \mathrm{d} r_2 \mathrm{d} y \mathrm{d} s .
        \end{aligned}
    \end{equation}
    By Lemma~\ref{I0lem}, H\"older's inequality and Minkowski inequality,
    \begin{equation}\nonumber
        \begin{aligned}
            &\Phi_{L,t,\tau}^{(2)} \\
            &\leq \frac{ e^{4 |\alpha| T} }{L^4} \int_0^t \int_0^L \int_{[0,\tau \wedge s]^2} \int_{[0,L]^2} \| \sigma(u(r_1,z_1)) \|_4 \| D_{r_1,z_1} \sigma(u(s,y)) \|_4 \\
            &\quad \times \| \sigma(u(r_2,z_2)) \|_4 \| D_{r_2,z_2} \sigma(u(s,y)) \|_4 \ \mathrm{d} z_1 \mathrm{d} z_2 \mathrm{d} r_1 \mathrm{d} r_2 \mathrm{d} y \mathrm{d} s \\
            &\leq \frac{ e^{4 |\alpha| T} M_{\sigma}^2 (1 + c_{T,4})^2  }{L^4} \int_0^t \int_0^L \int_{[0,\tau \wedge s]^2} \int_{[0,L]^2} \| D_{r_1,z_1} \sigma(u(s,y)) \|_4 \\
            &\quad \times \| D_{r_2,z_2} \sigma(u(s,y)) \|_4 \ \mathrm{d} z_1 \mathrm{d} z_2 \mathrm{d} r_1 \mathrm{d} r_2 \mathrm{d} y \mathrm{d} s ,
        \end{aligned}
    \end{equation}
    where $M_{\sigma}$ is the constant in the linear growth condition for $\sigma$.
    From the chain rule of Malliavin derivative for a Lipschitz function~\cite{nualart}, we obtain
    \begin{equation}\nonumber
        \begin{aligned}
            &\Phi_{L,t,\tau}^{(2)} \\
            &\leq \frac{ e^{4 |\alpha| T} M_{\sigma}^2 (1 + c_{T,4})^2  }{L^4} \int_0^t \int_0^L \int_{[0,\tau \wedge s]^2} \int_{[0,L]^2} \| G_{\sigma , u(s,y)} D_{r_1,z_1} u(s,y) \|_4 \\
            &\quad \times \| \sigma'(u(s,y)) D_{r_2,z_2} u(s,y) \|_4 \ \mathrm{d} z_1 \mathrm{d} z_2 \mathrm{d} r_1 \mathrm{d} r_2 \mathrm{d} y \mathrm{d} s \\
            &\leq \frac{ e^{4 |\alpha| T} M_{\sigma}^2 (1 + c_{T,4})^2 K_{\sigma}^2  }{L^4} \int_0^t \int_0^L \int_{[0,\tau \wedge s]^2} \int_{[0,L]^2} \| D_{r_1,z_1} u(s,y) \|_4 \\
            &\quad \times \| D_{r_2,z_2} u(s,y) \|_4 \ \mathrm{d} z_1 \mathrm{d} z_2 \mathrm{d} r_1 \mathrm{d} r_2 \mathrm{d} y \mathrm{d} s ,
        \end{aligned}
    \end{equation}
    where $K_{\sigma}$ is Lipschitz constant of $\sigma$ and $G_{\sigma , u(s,y)}$ is a random variable bounded by $K_{\sigma}$.
    In the case of Neumann/Dirichlet boundary conditions, applying Lemma~\ref{mdup}, we have
    \begin{equation}\nonumber
        \begin{aligned}
            &\Phi_{L,t,\tau}^{(2)} \\
            &\leq \frac{ e^{4 |\alpha| T} M_{\sigma}^2 (1 + c_{T,4})^2 K_{\sigma}^2 C_{T,4}^2 }{L^4} \int_0^t \int_0^L \int_{[0,\tau \wedge s]^2} \int_{[0,L]^2} \\
            &\quad \times p_{\beta(s - r_1)}(y-z_1) p_{\beta(s - r_2)}(y - z_2) \ \mathrm{d} z_1 \mathrm{d} z_2 \mathrm{d} r_1 \mathrm{d} r_2 \mathrm{d} y \mathrm{d} s \\
            &\leq \frac{ e^{4 |\alpha| T} M_{\sigma}^2 (1 + c_{T,4})^2 K_{\sigma}^2 C_{T,4}^2 }{L^4} \int_0^t \int_{[0,\tau \wedge s]^2} \int_{[0,L]^2} \int_{\mathbb{R}} \\
            &\quad \times p_{\beta(s - r_1)}(y-z_1) p_{\beta(s - r_2)}(y - z_2) \ \mathrm{d} y \mathrm{d} z_1 \mathrm{d} z_2 \mathrm{d} r_1 \mathrm{d} r_2 \mathrm{d} s \\
            &= \frac{ e^{4 |\alpha| T} M_{\sigma}^2 (1 + c_{T,4})^2 K_{\sigma}^2 C_{T,4}^2 }{L^4} \int_0^t \int_{[0,\tau \wedge s]^2} \int_{[0,L]^2} \\
            &\quad \times p_{\beta(2s - r_1 - r_2)}(z_1 - z_2)  \ \mathrm{d} z_1 \mathrm{d} z_2 \mathrm{d} r_1 \mathrm{d} r_2 \mathrm{d} s ,
        \end{aligned}
    \end{equation}
    where the equality holds by the semigroup property of Gaussian heat kernel $p_t(z)$.
    By using the identity \eqref{ptid}, we obtain
    \begin{equation}\nonumber
        \begin{aligned}
            &\Phi_{L,t,\tau}^{(2)} \\
            &\leq \frac{ e^{4 |\alpha| T} M_{\sigma}^2 (1 + c_{T,4})^2 K_{\sigma}^2 C_{T,4}^2 }{L^4} \int_0^t \int_{[0,\tau \wedge s]^2} \int_0^L \int_{\mathbb{R}} \\
            &\quad \times p_{\beta(2s - r_1 - r_2)}(z_1 - z_2)  \ \mathrm{d} z_1 \mathrm{d} z_2 \mathrm{d} r_1 \mathrm{d} r_2 \mathrm{d} s \\
            &= \frac{ e^{4 |\alpha| T} M_{\sigma}^2 (1 + c_{T,4})^2 K_{\sigma}^2 C_{T,4}^2 }{L^3} \int_0^t \int_{[0,\tau \wedge s]^2} \mathrm{d} r_1 \mathrm{d} r_2 \mathrm{d} s \\
            &\leq \frac{ T^3 e^{4 |\alpha| T} M_{\sigma}^2 (1 + c_{T,4})^2 K_{\sigma}^2 C_{T,4}^2 }{L^3}
        \end{aligned}
    \end{equation}
    In the case of periodic boundary conditions, applying Lemma~\ref{mdup} and identity \eqref{gid}, we have
    \begin{equation}\nonumber
        \begin{aligned}
            &\Phi_{L,t,\tau}^{(2)} \\
            &\leq \frac{ e^{4 |\alpha| T} M_{\sigma}^2 (1 + c_{T,4})^2 K_{\sigma}^2 C_{T,4}^2 }{L^4} \int_0^t \int_0^L \int_{[0,\tau \wedge s]^2} \int_{[0,L]^2} \\
            &\quad \times G_{s-r_1}(y,z_1) G_{s-r_2}(y,z_2) \ \mathrm{d} z_1 \mathrm{d} z_2 \mathrm{d} r_1 \mathrm{d} r_2 \mathrm{d} y \mathrm{d} s \\
            &= \frac{ e^{4 |\alpha| T} M_{\sigma}^2 (1 + c_{T,4})^2 K_{\sigma}^2 C_{T,4}^2 }{L^3} \int_0^t \int_{[0,\tau \wedge s]^2} e^{-\alpha(2s - r_1 - r_2)} \ \mathrm{d} r_1 \mathrm{d} r_2 \mathrm{d} s \\
            &\leq \frac{ T^3 e^{6 |\alpha| T} M_{\sigma}^2 (1 + c_{T,4})^2 K_{\sigma}^2 C_{T,4}^2 }{L^3} .
        \end{aligned}
    \end{equation}
    
    This completes the proof.
\end{proof}

We are now ready to prove Theorem~\ref{thm1}.

\begin{proof}[Proof of Theorem~\ref{thm1}]
    Applying Proposition~\ref{keyprop} with $t = \tau$, we have
    \begin{equation}\nonumber
        \sup_{t \in [0,T]} \var \left( \left\langle_{} D F_L(t), v_{L,t} \right\rangle_{\mathcal{H}} \right) \leq \frac{A_T}{L^3} 
    \end{equation}
    for all $L \geq 1$.
    Then, by using \eqref{flt2} and Proposition~\ref{malliavin-stein1},
    \begin{equation}
        \label{dtv1}
        \begin{aligned}
            &d_{\mathrm{TV}}\left( \frac{F_L(t)}{\sqrt{\var(F_L(t))}} , \mathcal{N}(0,1) \right) \\
            &\leq 2 \sqrt{\var \left( \left\langle_{} \frac{D F_L(t)}{\sqrt{\var(F_L(t))}} , \frac{ v_{L,t} }{\sqrt{\var(F_L(t))}} \right\rangle_{\mathcal{H}} \right) } \\
            &\leq \frac{2 \sqrt{A_t}}{ L^{3/2} \var(F_L(t)) } .
        \end{aligned}
    \end{equation}
    From Proposition~\ref{asymp_cov_low}, the asymptotic behavior of the variance is given by
    \begin{equation}\nonumber
        \var(F_L(t)) \gtrsim \frac{1}{L} \frac{\sigma(1)^2}{2} \int_0^{t \wedge \delta} e^{-2\alpha (t-s)} \ \mathrm{d} s,
    \end{equation}
    as $L \to \infty$. Consequently, using this result along with $\sigma(1) \neq 0$ and \eqref{dtv1}, we can deduce \eqref{d_conv}.
\end{proof}

\section{Proof of Theorem~\ref{thm2}}
We prepare the following proposition for the proof of Theorem~\ref{thm2}.
\begin{proposition}\label{prop_momineq}
    For every $T > 0$ and $k \geq 2$, there exists $A_{T,k} > 0$ such that for all $t_1, t_2 \in [0,T]$,
    \begin{equation}\label{momentineq}
        \| F_L(t_2) - F_L(t_1) \|_k \leq A_{T,k} |t_2 - t_1|^{1/2} L^{-1/2}
    \end{equation}
    uniformly for all $L \geq 1$.
\end{proposition}

\begin{proof}
    From \eqref{flt},
    \begin{equation}\label{flt3}
        F_L(t) = \frac{1}{L} \int_0^t \int_0^L \mathcal{I}_0(t-s,y) \sigma(u(s,y)) \ W(\mathrm{d}s, \mathrm{d} y) .
    \end{equation}
    We assume that $t_1 \leq t_2$.
    Applying Burkholder's inequality, for all $k \geq 2$,
    \begin{equation}\nonumber
        \begin{aligned}
            &\| F_L(t_2) - F_L(t_1) \|_k^2 \\
            &\leq \frac{2 z_k^2}{L^2} \left( \int_{t_1}^{t_2} \int_0^L \mathcal{I}_0^2(t_2 - s, y) \| \sigma(u(s,y)) \|_k^2 \ \mathrm{d} y \mathrm{d} s \right. \\
            &\quad + \left. \int_0^{t_1} \int_0^L ( \mathcal{I}_0(t_2 - s, y) - \mathcal{I}_0(t_1 - s, y) )^2 \| \sigma(u(s,y)) \|_k^2 \ \mathrm{d} y \mathrm{d} s  \right) ,
        \end{aligned}
    \end{equation}
    where $z_k$ is the constant in Burkholder's inequality.
    By using \eqref{momentup}, Minkowski's inequality and the change of variables,
    \begin{equation}\label{fl2fl1}
        \begin{aligned}
            &\| F_L(t_2) - F_L(t_1) \|_k^2 \\
            &\leq \frac{2 z_k^2 M_{\sigma}^2 (1 + c_{T,k})^2 }{L^2} \left( \int_{t_1}^{t_2} \int_0^L \mathcal{I}_0^2(t_2 - s, y) \ \mathrm{d} y \mathrm{d} s \right. \\
            &\quad + \left. \int_0^{t_1} \int_0^L ( \mathcal{I}_0(t_2 - s, y) - \mathcal{I}_0(t_1 - s, y) )^2 \ \mathrm{d} y \mathrm{d} s  \right) \\
            &= \frac{2 z_k^2 M_{\sigma}^2 (1 + c_{T,k})^2 }{L^2} \left( \int_{t_1}^{t_2} \int_0^L \mathcal{I}_0^2(t_2 - s, y) \ \mathrm{d} y \mathrm{d} s \right. \\
            &\quad + \left. \int_0^{t_1} \int_0^L ( \mathcal{I}_0(t_2 - t_1 + s, y) - \mathcal{I}_0(s, y) )^2 \ \mathrm{d} y \mathrm{d} s  \right) ,
        \end{aligned}
    \end{equation}
    where $M_{\sigma}$ is the constant in the linear growth condition for $\sigma$.
    In the case of Neumann/periodic boundary conditions, using $\mathcal{I}_0(t,y) = e^{-\alpha t}$(see Lemma~\ref{I0lem} and \eqref{gid} in Lemma~\ref{a1}),
    \begin{equation}\nonumber
        \begin{aligned}
            &\| F_L(t_2) - F_L(t_1) \|_k^2 \\
            &\leq \frac{2 z_k^2 M_{\sigma}^2 (1 + c_{T,k})^2 }{L^2} \left( \int_{t_1}^{t_2} \int_0^L e^{-2\alpha (t_2 - s)} \ \mathrm{d} y \mathrm{d} s \right. \\
            &\quad + \left. \int_0^{t_1} \int_0^L ( e^{-\alpha (t_2 - t_1 + s)} - e^{-\alpha s} )^2 \ \mathrm{d} y \mathrm{d} s  \right) .
        \end{aligned}
    \end{equation}
    From the following inequality
    \begin{equation}\nonumber
        | e^{-\alpha(t_2 - t_1)} - 1 | \leq |\alpha| e^{|\alpha| T } |t_2 - t_1| \quad \mathrm{for} \ t_1, t_2 \in [0,T] ,
    \end{equation}
    we have
    \begin{equation}\nonumber
        \begin{aligned}
            &\| F_L(t_2) - F_L(t_1) \|_k^2 \\
            &\leq \frac{2 z_k^2 M_{\sigma}^2 (1 + c_{T,k})^2 }{L^2} \left( \int_{t_1}^{t_2} \int_0^L e^{2 | \alpha | T} \ \mathrm{d} y \mathrm{d} s \right. \\
            &\quad + \left. \int_0^{t_1} \int_0^L e^{4|\alpha| T} |\alpha|^2 |t_2 - t_1|^2  \ \mathrm{d} y \mathrm{d} s  \right) \\
            &\leq \frac{2 z_k^2 M_{\sigma}^2 (1 + c_{T,k})^2 }{L} \left( e^{2 | \alpha | T} |t_2 - t_1| + T e^{4|\alpha| T} |\alpha|^2 |t_2 - t_1|^2 \right) \\
            &\leq \frac{2 z_k^2 M_{\sigma}^2 (1 + c_{T,k})^2 }{L} \left( e^{2 | \alpha | T} + T^2 e^{4|\alpha| T} |\alpha|^2 \right) |t_2 - t_1| .
        \end{aligned}
    \end{equation}
    In the case of Dirichlet boundary conditions, using the representation of Dirichlet heat kernel \eqref{dirichlet_rep},
    \begin{equation}\nonumber
        \begin{aligned}
            &\mathcal{I}_0(t_2 - t_1 + s, y) - \mathcal{I}_0(s,y) \\
            &= \frac{2}{L} \sum_{n=1}^{\infty} \sin \left( \frac{n \pi y}{L} \right) \int_0^L \sin \left( \frac{n \pi z}{L} \right) \ \mathrm{d} z \left[ e^{-\alpha(t_2 - t_1 + s)}e^{-\frac{n^2 \pi^2 (t_2 - t_1 + s)}{2L^2}} - e^{-\alpha s} e^{-\frac{n^2 \pi^2 s}{2L^2}} \right] \\
            &= \frac{2 e^{-\alpha s} }{L} \sum_{n=1}^{\infty} \sin \left( \frac{n \pi y}{L} \right) \frac{1 - \cos(n\pi)}{n\pi} e^{-\frac{n^2 \pi^2 s}{2L^2}} \left[ e^{-\alpha(t_2 - t_1)}e^{-\frac{n^2 \pi^2 (t_2 - t_1)}{2L^2}} - 1 \right] .
        \end{aligned}
    \end{equation}
    Then, applying the $L^2([0,L])$-orthogonality of the functions $\{y \mapsto \sin (n \pi y / L) \}_{n=1}^{\infty}$, we obtain
    \begin{equation}\nonumber
        \begin{aligned}
            &\int_0^{t_1} \int_0^L ( \mathcal{I}_0(t_2 - t_1 + s, y) - \mathcal{I}_0(s, y) )^2 \ \mathrm{d} y \mathrm{d} s \\
            &= 4 \int_0^{t_1} e^{-2\alpha s} \sum_{n=1}^{\infty} \int_0^L \sin^2 \left( \frac{n\pi y}{L} \right) \ \mathrm{d} y \left( \frac{1 - \cos(n\pi)}{n\pi} \right)^2 e^{-\frac{n^2 \pi^2 s}{L^2}} \left[ e^{-\alpha(t_2 - t_1)}e^{-\frac{n^2 \pi^2 (t_2 - t_1)}{2L^2}} - 1 \right]^2 \ \mathrm{d} s \\
            &= 2L \int_0^{t_1} e^{-2\alpha s} \sum_{n=1}^{\infty} \left( \frac{1 - \cos(n\pi)}{n\pi} \right)^2 e^{-\frac{n^2 \pi^2 s}{L^2}} \left[ e^{-\alpha(t_2 - t_1)}e^{-\frac{n^2 \pi^2 (t_2 - t_1)}{2L^2}} - 1 \right]^2 \ \mathrm{d} s \\
            &\leq 8L \int_0^{t_1} e^{-2\alpha s} \sum_{n=1}^{\infty} \frac{1}{n^2 \pi^2} e^{-\frac{n^2 \pi^2 s}{L^2}} \left[ e^{-\alpha(t_2 - t_1)}e^{-\frac{n^2 \pi^2 (t_2 - t_1)}{2L^2}} - 1 \right]^2 \ \mathrm{d} s .
        \end{aligned}
    \end{equation}
    By using the following inequality
    \begin{equation}\nonumber
        \begin{aligned}
            &\left| e^{-\alpha(t_2 - t_1)}e^{-\frac{n^2 \pi^2 (t_2 - t_1)}{2L^2}} - 1 \right| \\
            &\leq \left| e^{-\alpha(t_2 - t_1)}e^{-\frac{n^2 \pi^2 (t_2 - t_1)}{2L^2}} - e^{-\frac{n^2 \pi^2 (t_2 - t_1)}{2L^2}} \right| + \left| e^{-\frac{n^2 \pi^2 (t_2 - t_1)}{2L^2}} - 1 \right| \\
            &= e^{-\frac{n^2 \pi^2 (t_2 - t_1)}{2L^2}} \left| e^{-\alpha(t_2 - t_1)} - 1 \right| + \left| e^{-\frac{n^2 \pi^2 (t_2 - t_1)}{2L^2}} - 1 \right| \\
            &\leq |\alpha| e^{|\alpha| T} |t_2 - t_1| + \left| e^{-\frac{n^2 \pi^2 (t_2 - t_1)}{2L^2}} - 1 \right| \quad \mathrm{for} \ t_1, t_2 \in [0,T] ,
        \end{aligned}
    \end{equation}
    we have
    \begin{equation}\nonumber
        \begin{aligned}
            &\int_0^{t_1} \int_0^L ( \mathcal{I}_0(t_2 - t_1 + s, y) - \mathcal{I}_0(s, y) )^2 \ \mathrm{d} y \mathrm{d} s \\
            &\leq 16L e^{2|\alpha| T} \sum_{n=1}^{\infty} \frac{1}{n^2 \pi^2} \int_0^{t_1} e^{-\frac{n^2 \pi^2 s}{L^2}} \ \mathrm{d} s \left( |\alpha|^2 e^{2|\alpha| T} |t_2 - t_1|^2 + \left| e^{-\frac{n^2 \pi^2 (t_2 - t_1)}{2L^2}} - 1 \right|^2 \right) \\
            &= \frac{16}{6} L |\alpha|^2 T e^{4|\alpha| T} |t_2 - t_1|^2 + 16L e^{2|\alpha| T} \sum_{n=1}^{\infty} \frac{1}{n^2 \pi^2} \int_0^{t_1} e^{-\frac{n^2 \pi^2 s}{L^2}} \ \mathrm{d} s \left| e^{-\frac{n^2 \pi^2 (t_2 - t_1)}{2L^2}} - 1 \right|^2 .
        \end{aligned}
    \end{equation}
    Using the fact that $1 - e^{-x} \leq 1 \wedge x$ for all $x \geq 0$,
    \begin{equation}\nonumber
        \begin{aligned}
            &16L e^{2|\alpha| T} \sum_{n=1}^{\infty} \frac{1}{n^2 \pi^2} \int_0^{t_1} e^{-\frac{n^2 \pi^2 s}{L^2}} \ \mathrm{d} s \left| e^{-\frac{n^2 \pi^2 (t_2 - t_1)}{2L^2}} - 1 \right|^2 \\
            &= 16L e^{2|\alpha| T} \sum_{n=1}^{\infty} \frac{1}{n^2 \pi^2} \frac{L^2}{n^2 \pi^2} \left( 1 - e^{-\frac{n^2 \pi^2 t_1}{L^2}} \right) \left| e^{-\frac{n^2 \pi^2 (t_2 - t_1)}{2L^2}} - 1 \right|^2 \\
            &\leq 16L e^{2|\alpha| T} \sum_{n=1}^{\infty} \frac{1}{n^2 \pi^2} \left| 1 \wedge \frac{n^2 \pi^2 (t_2 - t_1)}{2 L^2} \right|^2 \frac{L^2}{n^2 \pi^2} \\
            &\leq \frac{16 e^{2|\alpha| T}}{L} \sum_{n=1}^{\infty} \left( \frac{L^4}{n^4 \pi^4} \wedge |t_2 - t_1|^2 \right) .
        \end{aligned}
    \end{equation}
    Moreover, we obtain
    \begin{equation}\nonumber
        \begin{aligned}
            &\frac{16 e^{2|\alpha| T}}{L} \sum_{n=1}^{\infty} \left( \frac{L^4}{n^4 \pi^4} \wedge |t_2 - t_1|^2 \right) \\
            &\leq \frac{16 e^{2|\alpha| T}}{L} \left( \sum_{n \leq |t_2 - t_1|^{-1/2} L / \pi} |t_2 - t_1|^2 + \frac{L^4}{\pi^4} \sum_{n > |t_2 - t_1|^{-1/2} L / \pi} \frac{1}{n^4} \right) \\
            &\leq \frac{16 e^{2|\alpha| T}}{L} \left( \frac{L}{\pi} |t_2 - t_1|^{3/2} + \frac{L^4}{\pi^4} \int_{|t_2 - t_1|^{-1/2} L / (2\pi)}^{\infty} \frac{1}{y^4} \ \mathrm{d} y \right) \\
            &= \frac{176}{3\pi} e^{2|\alpha| T} |t_2 - t_1|^{3/2} .
        \end{aligned}
    \end{equation}
    Then, we have
    \begin{equation}\nonumber
        \begin{aligned}
            &\int_0^{t_1} \int_0^L ( \mathcal{I}_0(t_2 - t_1 + s, y) - \mathcal{I}_0(s, y) )^2 \ \mathrm{d} y \mathrm{d} s \\
            &\leq \frac{16}{6} L |\alpha|^2 T e^{4|\alpha| T} |t_2 - t_1|^2 + \frac{176}{3\pi} e^{2|\alpha| T} |t_2 - t_1|^{3/2} \\
            &\leq \left( \frac{16}{6} L |\alpha|^2 T^2 e^{4|\alpha| T} + \frac{176}{3\pi} T^{1/2} e^{2|\alpha| T} \right) |t_2 - t_1| .
        \end{aligned}
    \end{equation}
    Hence, from \eqref{fl2fl1} and Lemma~\ref{I0lem}, we obtain
    \begin{equation}\nonumber
        \begin{aligned}
            &\| F_L(t_2) - F_L(t_1) \|_k^2 \\
            &\leq \frac{2 z_k^2 M_{\sigma}^2 (1 + c_{T,k})^2 }{L^2} \left( \int_{t_1}^{t_2} \int_0^L \mathcal{I}_0^2(t_2 - s, y) \ \mathrm{d} y \mathrm{d} s \right. \\
            &\quad + \left. \int_0^{t_1} \int_0^L ( \mathcal{I}_0(t_2 - t_1 + s, y) - \mathcal{I}_0(s, y) )^2 \ \mathrm{d} y \mathrm{d} s  \right) \\
            &\leq \frac{2 z_k^2 M_{\sigma}^2 (1 + c_{T,k})^2 }{L^2} \left( L e^{2|\alpha| T} |t_2 - t_1| + \left( \frac{16}{6} L |\alpha|^2 T^2 e^{4|\alpha| T} + \frac{176}{3\pi} T^{1/2} e^{2|\alpha| T} \right) |t_2 - t_1|  \right) \\
            &\leq \frac{2 z_k^2 M_{\sigma}^2 (1 + c_{T,k})^2 }{L} \left( e^{2|\alpha| T} + \frac{16}{6} |\alpha|^2 T^2 e^{4|\alpha| T} + \frac{176}{3\pi} T^{1/2} e^{2|\alpha| T} \right) |t_2 - t_1| .
        \end{aligned}
    \end{equation}
    Therefore, in the case of Neumann, Dirichlet, or periodic boundary conditions, we have \eqref{momentineq}.
    This completes the proof.
\end{proof}

\begin{proof}[Proof of Theorem~\ref{thm2}]
    The tightness of the family of processes $\{ \sqrt{L} F_L(\cdot) \}_{L \geq 1}$ in $C([0,T])$ is guaranteed by Proposition~\ref{prop_momineq} (see, e.g.,~\cite{karatzas_shreve}).
    What remains is to show that the finite-dimensional distributions of $( \sqrt{L} F_L(t) )_{t \in [0,T]}$ converge to those of $(\int_0^t e^{-\alpha(t-s)} \sqrt{f_{\sigma}(s)} \ \mathrm{d} W_s)_{t \in [0,T]}$ as $L \to \infty$
    
    Let us fix $T > 0$ and $m \geq 1$ points $t_1, \dots, t_m \in (0,T]$.
    According to Proposition~\ref{asymp_cov}, as $L \to \infty$, the covariance satisfies
    \begin{equation}\label{covsim}
        \cov\left[ F_L(t_i), F_L(t_j) \right] \sim \frac{1}{L} \int_0^{t_i \wedge t_j} e^{-\alpha(t_i + t_j - 2s)} f_{\sigma}(s) \ \mathrm{d} s
    \end{equation}
    for all $i,j = 1, \dots, m$.
    To proceed, let us define the vector $F := (F_1, \dots, F_m)$ with components
    \begin{equation}\nonumber
        F_i := \frac{F_L(t_i)}{\sqrt{\var(F_L(t_i))}}.
    \end{equation}
    We then define $G = (G_1, \dots, G_m)$ to be a centered Gaussian random vector whose covariance matrix $C = (C_{i,j})$ is given by
    \begin{equation}\nonumber
        C_{i,j} := \cov \left[ F_i , F_j \right].
    \end{equation}

    Let us introduce the rescaled random fields $V_1, \dots, V_m$ by setting
    \begin{equation}\nonumber
        V_i := \frac{v_{L,t_i}}{\sqrt{\var(F_L(t_i))}}, \quad i=1, \dots, m,
    \end{equation}
    where the fields $v_{L,t_i}$ are given in \eqref{vlt}.
    From \eqref{flt2}, we have the relation $F_i = \delta(V_i)$.
    By duality, $E[ \langle D F_i , V_j \rangle_{\mathcal{H}} ] = E[ F_i \delta(V_j) ] = C_{i,j}$ for all $i, j = 1, \dots, m$.
    Applying Proposition~\ref{malliavin-stein2} then yields the desired bound for any $h \in C^2(\mathbb{R}^m)$:
    \begin{equation}\nonumber
        \left| E[h(F)] - E[h(G)] \right| \leq \frac{1}{2} \| h'' \|_{\infty} \sqrt{ \sum_{i,j=1}^m \var\left( \langle D F_i , V_j \rangle_{\mathcal{H}} \right) }.
    \end{equation}
    From Proposition~\ref{keyprop},
    \begin{equation}\nonumber
        \begin{aligned}
            \var\left( \langle D F_i , V_j \rangle_{\mathcal{H}} \right) &= \frac{ \var\left( \langle D F_L(t_i) , v_{L,t_j} \rangle_{\mathcal{H}} \right) }{ \var( F_L(t_i) ) \var( F_L(t_j) ) } \\
            &\leq \frac{A_T}{L^3 \var( F_L(t_i) ) \var( F_L(t_j) )} ,
        \end{aligned}
    \end{equation}
    which together with \eqref{covsim} implies that
    \begin{equation}\nonumber
        \lim_{L \to \infty} \left| E[h(F)] - E[h(G)] \right| = 0
    \end{equation}
    for all $h \in C^2(\mathbb{R}^m)$.

    We also examine the limit of the covariances $C_{i,j}$. Using the result from \eqref{covsim}, we find that as $L \to \infty$,
    \begin{equation}\nonumber
        C_{i,j} \to \frac{ \int_0^{t_i \wedge t_j} e^{-\alpha(t_i + t_j - 2s)} f_{\sigma}(s) \ \mathrm{d} s }{ \sqrt{\int_0^{t_i} e^{-2\alpha(t_i - s)} f_{\sigma}(s) \ \mathrm{d} s } \sqrt{\int_0^{t_j} e^{-2\alpha(t_j - s)} f_{\sigma}(s) \ \mathrm{d} s } }.
    \end{equation}
    It follows that the law of the Gaussian vector $G$, which is determined by its covariance structure, converges weakly to that of
    \begin{equation}\label{G_conv}
        \begin{aligned}
            \left( \frac{ \int_0^{t_1} e^{-\alpha(t_1 - s)} \sqrt{f_{\sigma}(s)} \ \mathrm{d} W_s }{ \sqrt{ \int_0^{t_1} e^{-2\alpha (t_1 - s)} f_{\sigma}(s) \ \mathrm{d} s } } , \dots, \frac{ \int_0^{t_m} e^{-\alpha(t_m - s)} \sqrt{f_{\sigma}(s)} \ \mathrm{d} W_s }{ \sqrt{ \int_0^{t_m} e^{-2\alpha (t_m - s)} f_{\sigma}(s) \ \mathrm{d} s } } \right).
        \end{aligned}
    \end{equation}

    Therefore, $F$ converges weakly to the random vector in \eqref{G_conv} as $L \to \infty$.
    Recalling the definition of $F_i$ and using the asymptotic variance from \eqref{covsim}, the random vector
    \begin{equation}\nonumber
        \sqrt{L} \left( \frac{ F_L(t_1) }{ \sqrt{ \int_0^{t_1} e^{-2\alpha (t_1 - s)} f_{\sigma}(s) \ \mathrm{d} s  } } , \dots, \frac{ F_L(t_m) }{ \sqrt{ \int_0^{t_m} e^{-2\alpha (t_m - s)} f_{\sigma}(s) \ \mathrm{d} s  } } \right) 
    \end{equation}
    converges to the random vector in \eqref{G_conv} as $L \to \infty$.
    This completes the proof.
\end{proof}

\section{On the limit in Assumption~\ref{assump}}
\label{on_assump}
If $\sigma(u) = \sigma_1 u + \sigma_0$ for some constants $\sigma_1$ and $\sigma_0$, we can calculate the limit $f_{\sigma}(t)$ in Assumption~\ref{assump}.
Our calculation is based on the Wiener chaos decomposition of $u(t,x)$.
Let 
\begin{equation}\nonumber
    u_0(t,x) = \int_0^L G_t(x,y) \ \mathrm{d} y
\end{equation}
for every $(0,T] \times [0,L]$ and $u_0(0,x) = u_0(x) = 1$ for all $x \in [0,L]$.
Let
\begin{equation}\nonumber
    \mathcal{I}_0(t,x) := u_0(t,x)
\end{equation}
and
\begin{equation}\nonumber
    \begin{aligned}
        \mathcal{I}_k(t,x) := &\sigma_1^{k-1} \int_0^t \int_0^L \dots \int_0^{r_{k-1}} \int_0^L G_{t-r_1}(x,z_1) \dots G_{r_{k-1} - r_k}(z_{k-1}, z_k) \\
        &\qquad \qquad \times (\sigma_1 \mathcal{I}_0(r_k, z_k) + \sigma_0) \ W(\mathrm{d}r_k, \mathrm{d}z_k) \dots W(\mathrm{d}r_1, \mathrm{d}z_1)
    \end{aligned}
\end{equation}
for all $k \geq 1$.
We define the Picard iteration $\{ u_n(t,x) \}_{n=0}^{\infty}$ for $u(t,x)$.
Define iteratively, for every $n \geq 0$,
\begin{equation}\label{picard}
    u_{n+1}(t,x) := u_0(t,x) + \int_0^t \int_0^L G_{t-r}(x,z) \sigma(u_n(r,z)) \ W(\mathrm{d}r, \mathrm{d}z) .
\end{equation}
Since we only consider the case $\sigma(u) = \sigma_1 u + \sigma_0$ through this section, we have
\begin{equation}\nonumber
    \begin{aligned}
        &u_{n+1}(t,x) \\
        &= u_0(t,x) + \int_0^t \int_0^L G_{t-r_1}(x,z_1) (\sigma_1 u_n(r_1,z_1) + \sigma_0) \ W(\mathrm{d}r_1, \mathrm{d}z_1) .
    \end{aligned}
\end{equation}
Then, by using mathematical induction, we obtain the Wiener chaos decomposition
\begin{equation}\label{un_chaos}
    u_{n+1}(t,x) = \sum_{k=0}^{n} \mathcal{I}_k(t,x) .
\end{equation}
Therefore, $u_n \to u$ in $L^p$ ($p \geq 2$) (see~\cite{walsh}) and applying \eqref{un_chaos}, we have
\begin{equation}\label{u_chaos}
    u(t,x) = \sum_{k=0}^{\infty} \mathcal{I}_k(t,x) .
\end{equation}
Moreover, by multiple It\^o's isometry,
\begin{equation}\label{mlito}
    \| u(t,x) \|_2^2 = \sum_{k=0}^{\infty} \| \mathcal{I}_k(t,x) \|_2^2
\end{equation}
where
\begin{equation}\nonumber
    \begin{aligned}
        \| \mathcal{I}_k(t,x) \|_2^2 = &\sigma_1^{2(k-1)} \int_0^t \int_0^L \dots \int_0^{r_{k-1}} \int_0^L G_{t-r_1}^2(x,z_1) \dots G_{r_{k-1} - r_k}^2(z_{k-1}, z_k) \\
        &\qquad \qquad \times (\sigma_1 \mathcal{I}_0(r_k, z_k) + \sigma_0)^2 \ \mathrm{d} r_k \mathrm{d} z_k \dots \mathrm{d} r_1 \mathrm{d} z_1
    \end{aligned}
\end{equation}
for all $k \geq 1$.

Using Lemma~\ref{I0lem}, we get the following Propositions.

\begin{proposition}
    \label{Iknd}
    Fix $T > 0$.
    In the case of Neumann/Dirichlet boundary conditions, for every $k \geq 1$
    \begin{equation}\label{Iknd1}
        \sup_{L \geq 1} \sup_{(t,x) \in [0,T] \times [0,L]} \| \mathcal{I}_k(t,x) \| \leq \sigma_1^{2(k-1)} (|\sigma_1| e^{|\alpha|T} + |\sigma_0| )^2 \frac{K_T^{2k} (4\beta)^{-k/2} T^{k/2}}{\Gamma((k+2)/2)} ,
    \end{equation}
    where  the constant $K_T$ is defined in \eqref{const_kt}.
    Moreover, for every $t > 0$ and $k \geq 1$, there exists $f_k(t)$ such that
    \begin{equation}\nonumber
        \lim_{L \to \infty} \frac{1}{L} \int_0^L \| \mathcal{I}_k(t,x) \|_2^2 \ \mathrm{d} x = f_k(t) .
    \end{equation}
\end{proposition}

\begin{proof}
    \begin{equation}\nonumber
        \begin{aligned}
            &\| \mathcal{I}_k(t,x) \|_2^2 \\
            &\leq \sigma_1^{2(k-1)} K_T^{2k} (|\sigma_1| e^{|\alpha|T} + |\sigma_0| )^2 \\
            &\quad \times \int_0^t \int_{\mathbb{R}} \dots \int_0^{r_{k-1}} \int_{\mathbb{R}} p_{\beta(t-r_1)}^2(x-z_1) \dots p_{\beta(r_{k-1} - r_k)}^2(z_{k-1} - z_k) \ \mathrm{d} r_k \mathrm{d} z_k \dots \mathrm{d} r_1 \mathrm{d} z_1 \\
            &= \sigma_1^{2(k-1)} K_T^{2k} (|\sigma_1| e^{|\alpha|T} + |\sigma_0| )^2 \left( \frac{t}{4\pi \beta} \right)^{k/2} \\
            &\quad \times \int_{0 < r_k < \dots < r_1 < 1} \sqrt{\frac{1}{(1-r_1)\times \dots \times (r_{k-1} - r_k)}} \ \mathrm{d} r_k \dots \mathrm{d} r_1 \\
            &= \sigma_1^{2(k-1)} K_T^{2k} (|\sigma_1| e^{|\alpha|T} + |\sigma_0| )^2 \left( \frac{t}{4\pi \beta} \right)^{k/2} \frac{\Gamma(1/2)^k}{\Gamma((k+2)/2)} \\
            &\leq \sigma_1^{2(k-1)} (|\sigma_1| e^{|\alpha|T} + |\sigma_0| )^2 \frac{K_T^{2k} (4\beta)^{-k/2} T^{k/2}}{\Gamma((k+2)/2)}
        \end{aligned}
    \end{equation}
    where the first equality follows from semigroup property of Gaussian heat kernel $p_t(z)$ and change of variables, and the second one holds by the following identity
    \begin{equation}\label{multigamma}
        \int_{0 < r_k < \dots < r_1 < 1} \sqrt{\frac{1}{(1-r_1)\times \dots \times (r_{k-1} - r_k)}} \ \mathrm{d} r_k \mathrm{d} z_k \dots \mathrm{d} r_1 \mathrm{d} z_1 = \frac{\Gamma(1/2)^k}{\Gamma((k+2)/2)} ,
    \end{equation}
    see~\cite{olver_lozier_boisvert_clark}*{5.14.1}.
    Then, we have \eqref{Iknd1}.

    Using the same arguments as~\cite{pu}*{Proposition 3.4.}, we obtain
    \begin{equation}\nonumber
        \begin{aligned}
            &\frac{1}{L} \int_0^L \| \mathcal{I}_k(t,x) \|_2^2 \ \mathrm{d} x \\
            &= \sigma_1^{2(k-1)} \int_{0 < r_k < \dots < r_1 < t} e^{-2\alpha(t-r_1)} p_{2\beta(t-r_1)}(0) \dots e^{-2\alpha(r_{k-1} - r_k)} p_{2\beta(r_{k-1} - r_k)}(0) \\
            &\quad \times \frac{1}{L} \int_0^L (\sigma_1 \mathcal{I}_0(r_k, z_k) + \sigma_0)^2 \ \mathrm{d} z_k \mathrm{d} r_k \dots \mathrm{d} r_1 \\
            &\quad + o(L) .
        \end{aligned}
    \end{equation}
    Therefore, using dominated convergence theorem, \eqref{Iknd1} and Lemma~\ref{I0lem}, we have
    \begin{equation}\label{fk}
        \begin{aligned}
            &\lim_{L \to \infty} \frac{1}{L} \int_0^L \| \mathcal{I}_k(t,x) \|_2^2 \ \mathrm{d} x \\
            &= \sigma_1^{2(k-1)} \int_{0 < r_k < \dots < r_1 < t} e^{-2\alpha(t-r_1)} p_{2\beta(t-r_1)}(0) \dots e^{-2\alpha(r_{k-1} - r_k)} p_{2\beta(r_{k-1} - r_k)}(0) \\
            &\quad \times \lim_{L \to \infty} \frac{1}{L} \int_0^L (\sigma_1 \mathcal{I}_0(r_k, z_k) + \sigma_0)^2 \ \mathrm{d} z_k \mathrm{d} r_k \dots \mathrm{d} r_1 \\
            &= \sigma_1^{2(k-1)} \int_{0 < r_k < \dots < r_1 < t} e^{-2\alpha(t-r_1)} p_{2\beta(t-r_1)}(0) \dots e^{-2\alpha(r_{k-1} - r_k)} p_{2\beta(r_{k-1} - r_k)}(0) \\
            &\quad \times ( \sigma_1^2 e^{-2\alpha r_k} + 2 \sigma_1 \sigma_0 e^{-\alpha r_k} + \sigma_0^2 ) \ \mathrm{d} r_k \dots \mathrm{d} r_1 \\
            &= \sigma_1^{2(k-1)} e^{-2\alpha t} \int_{0 < r_k < \dots < r_1 < t} p_{2\beta(t-r_1)}(0) \dots p_{2\beta(r_{k-1} - r_k)}(0) \\
            &\quad \times ( \sigma_1 + e^{\alpha r_k} \sigma_0 )^2 \ \mathrm{d} r_k \dots \mathrm{d} r_1 \\
            &=: f_k(t) .
        \end{aligned}
    \end{equation}
\end{proof}

\begin{proposition}
    \label{Ikp}
    In the case of periodic boundary conditions, for all $k \geq 1$ and $t > 0$, $\| \mathcal{I}_k(t,x) \|_2^2$ does not depend on $x \in [0,L]$ and
    \begin{equation}\nonumber
        \lim_{L \to \infty} \| \mathcal{I}_k(t,x) \|_2^2 = f_k(t)
    \end{equation}
    where $f_k(t)$ is the limit in .
\end{proposition}

\begin{proof}
    Since $\mathcal{I}_0(t,x) = e^{-\alpha t}$(see Lemma~\ref{I0lem} and \eqref{gid} in Lemma~\ref{a1}), we have
    \begin{equation}\nonumber
    \begin{aligned}
        &\| \mathcal{I}_k(t,x) \|_2^2 \\
        &= \sigma_1^{2(k-1)} \int_0^t \int_0^L \dots \int_0^{r_{k-1}} \int_0^L G_{t-r_1}^2(x,z_1) \dots G_{r_{k-1} - r_k}^2(z_{k-1}, z_k) \\
        &\quad \times (\sigma_1 e^{-\alpha r_k} + \sigma_0)^2 \ \mathrm{d} r_k \mathrm{d} z_k \dots \mathrm{d} r_1 \mathrm{d} z_1 \\
        &= \sigma_1^{2(k-1)} \int_{0 < r_k < \dots < r_1 < t} G_{2(t-r_1)}(0,0) \dots G_{2(r_{k-1} - r_k)}(0,0) (\sigma_1 e^{-\alpha r_k} + \sigma_0)^2 \ \mathrm{d} r_1 \dots \mathrm{d} r_k
    \end{aligned}
\end{equation}
where in the second equality follows from semigroup property of $G$.
Then, it follows that $\| \mathcal{I}_k(t,x) \|_2^2$ does not depend on $x \in [0,L]$.
Moreover, by dominated convergence theorem,
\begin{equation}\nonumber
    \begin{aligned}
        &\lim_{L \to \infty} \| \mathcal{I}_k(t,x) \|_2^2 \\
        &= \sigma_1^{2(k-1)} \int_{0 < r_k < \dots < r_1 < t} e^{-2\alpha (t-r_1)} p_{2\beta(t-r_1)}(0) \dots e^{-2\alpha (r_{k-1} - r_k)} p_{2\beta(r_{k-1} - r_k)}(0) \\
        &\quad \times (\sigma_1 e^{-\alpha r_k} + \sigma_0)^2 \ \mathrm{d} r_1 \dots \mathrm{d} r_k \\
        &= \sigma_1^{2(k-1)} e^{-2\alpha t} \int_{0 < r_k < \dots < r_1 < t} p_{2\beta(t-r_1)}(0) \dots p_{2\beta(r_{k-1} - r_k)}(0) \\
        &\quad \times (\sigma_1 + \sigma_0 e^{\alpha r_k})^2 \ \mathrm{d} r_1 \dots \mathrm{d} r_k \\
        &= f_k(t) ,
    \end{aligned}
\end{equation}
where $f_k$ is the same as in Proposition~\ref{Iknd}.
\end{proof}

By Proposition~\ref{Iknd} and~\ref{Ikp}, we can check the Assumption~\ref{assump}.
\begin{proposition}
    For every $t > 0$,
    \begin{equation}\nonumber
        \begin{aligned}
            &\lim_{L \to \infty} \frac{1}{L} \int_0^L E[\sigma(u(t,x))^2] \ \mathrm{d} x \\
            &= \sigma_1^2 \sum_{k=0}^{\infty} f_k(t) + 2 \sigma_1 \sigma_0 e^{-\alpha t} + \sigma_0^2 \\
            &=: f_{\sigma}(t) ,
        \end{aligned}
    \end{equation}
    where $\sigma(u) = \sigma_1 u + \sigma_0$, $f_0(t) := e^{-2\alpha t}$ and $f_k$ is defined in \eqref{fk}.
    Moreover, 
    \begin{equation}\nonumber
        \begin{aligned}
            f_{\sigma}(t) &\leq e^{-2\alpha t} ( |\sigma_1| + e^{|\alpha| t } |\sigma_0| )^2 (f(\sigma_1^4 t / \beta) - 1) \\
            &\quad + (\sigma_1 e^{-\alpha t} + \sigma_0)^2 ,
        \end{aligned}
    \end{equation}
    where
    \begin{equation}\label{func_f}
        f(t) := 2 e^{t/4} \int_{-\infty}^{\sqrt{t/2}} \frac{1}{\sqrt{2\pi}} e^{-y^2 / 2} \ \mathrm{d} y .
    \end{equation}
    In particular, $f_{\sigma} \in L^1([0,T])$ for all $T > 0$.
\end{proposition}

\begin{proof}
    We first consider the Neumann/Dirichlet case.
    By proposition \eqref{mlito} and Tonelli's theorem,
    \begin{equation}\nonumber
        \frac{1}{L} \int_0^L E[u(t,x)^2] \ \mathrm{d} x = \sum_{k=0}^{\infty} \frac{1}{L} \int_0^L \| \mathcal{I}_k(t,x) \|_2^2 \ \mathrm{d} x .
    \end{equation}
    Since the series $\sum_{k=1}^{\infty} \sigma_1^{2(k-1)} (|\sigma_1| e^{|\alpha|T} + |\sigma_0| )^2 \frac{K_T^{2k} (4\beta)^{-k/2} T^{k/2}}{\Gamma((k+2)/2)}$ converges, by Proposition~\ref{Iknd}, dominated convergence theorem and Lemma~\ref{I0lem}, we have
    \begin{equation}\nonumber
        \begin{aligned}
            &\lim_{L \to \infty} \frac{1}{L} \int_0^L E[u(t,x)^2] \ \mathrm{d} x \\
            &= \sum_{k=0}^{\infty} \lim_{L \to \infty} \frac{1}{L} \int_0^L \| \mathcal{I}_k(t,x) \|_2^2 \ \mathrm{d} x \\
            &= e^{-2\alpha t} + \sum_{k=1}^{\infty} f_k(t) \\
            &= \sum_{k=0}^{\infty} f_k(t) ,
        \end{aligned}
    \end{equation}
    where $f_0(t) := e^{-2\alpha t}$.

    Similarly, in the case of periodic boundary conditions, Proposition~\ref{Ikp} and \eqref{mlito} imply that $E[u(t,x)^2]$ does not depend on $x \in [0,L]$.
    Hence, we have
    \begin{equation}\nonumber
        \begin{aligned}
            &\lim_{L \to \infty} \frac{1}{L} \int_0^L E[u(t,x)^2] \ \mathrm{d} x \\
            &= \lim_{L \to \infty} E[u(t,0)^2] \\
            &= \sum_{k=0}^{\infty} \lim_{L \to \infty} \| \mathcal{I}_k(t,0) \|_2^2 \\
            &= e^{-2\alpha t} + \sum_{k=1}^{\infty} f_k(t) \\
            &= \sum_{k=0}^{\infty} f_k(t) ,
        \end{aligned}
    \end{equation}
    where in the second equality, we apply dominated convergence theorem in order to exchange the limit and the sum.

    Therefore, in the case of Neumann, Dirichlet, or periodic boundary conditions, we obtain
    \begin{equation}\nonumber
        \lim_{L \to \infty} \frac{1}{L} \int_0^L E[u(t,x)^2] \ \mathrm{d} x = \sum_{k=0}^{\infty} f_k(t) .
    \end{equation}
    Hence, we have
    \begin{equation}\nonumber
        \begin{aligned}
            &\lim_{L \to \infty} \frac{1}{L} \int_0^L E[\sigma(u(t,x))^2] \ \mathrm{d} x \\
            &= \sigma_1^2 \lim_{L \to \infty} \frac{1}{L} \int_0^L E[u(t,x)^2] \ \mathrm{d} x \\
            &\quad + 2\sigma_1 \sigma_0 \lim_{L \to \infty} \frac{1}{L} \int_0^L E[u(t,x)] \ \mathrm{d} x + \sigma_0^2 \\
            &= \sigma_1^2 \sum_{k=0}^{\infty} f_k(t) + 2 \sigma_1 \sigma_0 e^{-\alpha t} + \sigma_0^2 ,
        \end{aligned}
    \end{equation}
    where the second equality holds by $E[u(t,x)] = \mathcal{I}_0(t,x)$ and Lemma~\ref{I0lem}.

    From \eqref{multigamma}, \eqref{fk} and change of variables, for all $k \geq 1$, we have
    \begin{equation}\nonumber
        \begin{aligned}
            \sigma_1^2 f_k(t) &\leq \sigma_1^{2k} e^{-2\alpha t} ( |\sigma_1| + e^{|\alpha| t } |\sigma_0| )^2  \\
            &\quad \times \int_{0 < r_k < \dots < r_1 < t} p_{2\beta(t-r_1)}(0) \dots p_{2\beta(r_{k-1} - r_k)}(0) \ \mathrm{d} r_k \dots \mathrm{d} r_1 \\
            &= e^{-2\alpha t} ( |\sigma_1| + e^{|\alpha| t } |\sigma_0| )^2 \frac{(\sigma_1^4 t/(4\beta))^{k/2}}{\Gamma((k+2)/2)} .
        \end{aligned}
    \end{equation}
    By using the following identity(see~\cite{chen}*{Lemma 2.3.4})
    \begin{equation}\nonumber
        \sum_{n=1}^{\infty} \frac{\lambda^{n-1}}{\Gamma((n+1)/2)} = 2 e^{\lambda^2} \int_{-\infty}^{\sqrt{2} \lambda} \frac{1}{\sqrt{2\pi}} e^{-y^2 / 2} \ \mathrm{d} y,\ \mathrm{for~all~} \lambda \geq 0
    \end{equation}
    with $\lambda = \sqrt{\sigma_1^4 t/(4 \beta)}$, we obtain
    \begin{equation}\nonumber
        \begin{aligned}
            &\sum_{k=0}^{\infty} \frac{(\sigma_1^4 t/(4\beta))^{k/2}}{\Gamma((k+2)/2)} \\
            &= \sum_{n=1}^{\infty} \frac{(\sigma_1^4 t/(4\beta))^{(n-1)/2}}{\Gamma((n+1)/2)} \\
            &= f(\sigma_1^4 t / \beta) ,
        \end{aligned}
    \end{equation}
    where $f$ is defined in \eqref{func_f}.
    Hence, we have
    \begin{equation}\nonumber
        \begin{aligned}
            &\sigma_1^2 \sum_{k=0}^{\infty} f_k(t) \\
            &= \sigma_1^2 e^{-2\alpha t} + \sum_{k=1}^{\infty} \sigma_1^2 f_k(t) \\
            &\leq \sigma_1^2 e^{-2\alpha t} + e^{-2\alpha t} ( |\sigma_1| + e^{|\alpha| t } |\sigma_0| )^2 \sum_{k=1}^{\infty} \frac{(\sigma_1^4 t/(4\beta))^{k/2}}{\Gamma((k+2)/2)} \\
            &= \sigma_1^2 e^{-2\alpha t} + e^{-2\alpha t} ( |\sigma_1| + e^{|\alpha| t } |\sigma_0| )^2 \sum_{k=0}^{\infty} \frac{(\sigma_1^4 t/(4\beta))^{k/2}}{\Gamma((k+2)/2)} \\
            &\quad - e^{-2\alpha t} ( |\sigma_1| + e^{|\alpha| t } |\sigma_0| )^2 \\
            &= \sigma_1^2 e^{-2\alpha t} + e^{-2\alpha t} ( |\sigma_1| + e^{|\alpha| t } |\sigma_0| )^2 (f(\sigma_1^4 t / \beta) - 1) .
        \end{aligned}
    \end{equation}
    Therefore, we get
    \begin{equation}\nonumber
        \begin{aligned}
            f_{\sigma}(t) &= \sigma_1^2 \sum_{k=0}^{\infty} f_k(t) + 2 \sigma_1 \sigma_0 e^{-\alpha t} + \sigma_0^2 \\
            &\leq e^{-2\alpha t} ( |\sigma_1| + e^{|\alpha| t } |\sigma_0| )^2 (f(\sigma_1^4 t / \beta) - 1) \\
            &\quad + (\sigma_1 e^{-\alpha t} + \sigma_0)^2 .
        \end{aligned}
    \end{equation}
    Then, $f_\sigma \in L^1([0,T])$ for all $T > 0$.
\end{proof}

\begin{appendices}
    
\section{Properties of  Green's function}
Denote the Gaussian heat kernel on $\mathbb{R}$ as
\begin{equation}\label{gaussian_heat_kernel}
    p_t(z) = \frac{1}{\sqrt{2\pi t}} e^{-\frac{z^2}{2t}} , \quad t > 0, z \in \mathbb{R} .
\end{equation}
Let $G$~(resp. $g$) be the Green's function for the cable equation(resp. heat equation) with Neumann/Dirichlet/periodic boundary conditions.
Note that
\begin{equation}\label{Gg}
    G_t(x,y) = e^{-\alpha t} g_{\beta t}(x,y) .
\end{equation}
Hence, the properties of $G$ follow directly from those of $g$.
For all $t > 0$ and $x, y \in [0,L]$, in the case of Neumann boundary conditions,
\begin{equation}\label{eq_guarantees4}
    G_t(x,y) = e^{-\alpha t} \sum_{n \in \mathbb{Z}} \left( p_{\beta t}(x - y + 2nL) + p_{\beta t}(x + y + 2nL) \right) ,
\end{equation}
or equivalently,
\begin{equation}\nonumber
    G_t(x,y) = \frac{e^{-\alpha t}}{\sqrt{L}} + \frac{2 e^{-\alpha t}}{L} \sum_{n=1}^{\infty} \cos \left( \frac{n \pi x}{L} \right) \cos \left( \frac{n \pi y}{L} \right) e^{-\frac{n^2 \pi^2 \beta t}{2 L^2}} ;
\end{equation}
and in the case of Dirichlet boundary conditions,
\begin{equation}\nonumber
    G_t(x,y) = e^{-\alpha t} \sum_{n \in \mathbb{Z}} \left( p_{\beta t}(x - y + 2nL) - p_{\beta t}(x + y + 2nL) \right) ,
\end{equation}
or equivalently,
\begin{equation}\label{dirichlet_rep}
    G_t(x,y) = \frac{2 e^{-\alpha t}}{L} \sum_{n=1}^{\infty} \sin \left( \frac{n \pi x}{L} \right) \sin \left( \frac{n \pi y}{L} \right) e^{-\frac{n^2 \pi^2 \beta t}{2 L^2}} ;
\end{equation}
and in the case of periodic boundary conditions,
\begin{equation}\label{periodic_rep}
    G_t(x,y) = e^{-\alpha t} \sum_{n \in \mathbb{Z}} p_{\beta t}(x - y + nL) .
\end{equation}
Moreover, by using the properties of $g$~\cite{pu}*{Lemma A.1.}, we get some useful properties of $G$.
\begin{lemma}
    \label{a1}
    \leavevmode
    \begin{itemize}
        \item[(1)] Symmetry. $G_t(x,y) = G_t(y,x)$ for all $t > 0$, $x, y \in [0,L]$.
        \item[(2)] In the case of Neumann and periodic boundary conditions, for all $t > 0$ and $x \in [0,L]$,
        \begin{equation}\label{gid}
            \int_0^L G_t(x,y) \ \mathrm{d} y = e^{-\alpha t} .
        \end{equation}
        \item[(3)] Semigroup property. For all $t, s > 0$ and $x, y \in [0,L]$,
        \begin{equation}\nonumber
            \int_0^L G_t(x,z) G_s(z,y) \ \mathrm{d} z = G_{t+s}(x,y) .
        \end{equation}
        \item[(4)] In the case of Neumann/Dirichlet boundary conditions, for every $t > 0$ and $x, y \in [0,L]$,
        \begin{equation}\nonumber
            G_t(x,y) \leq e^{-\alpha t} p_{\beta t}(x-y) \left( 4 + \frac{4}{1 - e^{-\frac{L^2}{\beta t}}} \right) ,
        \end{equation}
        and as a consequence, for all $t \in (0,T]$, $L \geq 1$ and $x, y \in [0,L]$,
        \begin{equation}\nonumber
            G_t(x,y) \leq K_T p_{\beta t}(x-y),
        \end{equation}
        where
        \begin{equation}\label{const_kt}
            K_T = e^{|\alpha| T} \left( 4 + \frac{4}{1 - e^{-\frac{1}{\beta T}}} \right) .
        \end{equation}
    \end{itemize}
\end{lemma}

\end{appendices}

\newpage
\begin{flushleft}
\mbox{  }\\
\hspace{95mm} Soma Nishino\\
\hspace{95mm} Department of Mathematical Sciences\\
\hspace{95mm} Tokyo Metropolitan University\\
\hspace{95mm} Hachioji, Tokyo 192-0397\\
\hspace{95mm} Japan\\
\hspace{95mm} e-mail: nishino-soma@ed.tmu.ac.jp\\
\end{flushleft}

\end{document}